\documentclass{amsart}
\usepackage{amsmath,amsthm,amssymb}
\usepackage{graphicx}

\newtheorem{theorem}{Theorem}[section]
\newtheorem{proposition}[theorem]{Proposition}
\newtheorem{lemma}[theorem]{Lemma}

\newtheorem{claim}[theorem]{Claim}

\theoremstyle{definition}
\newtheorem{definition}[theorem]{Definition}
\newtheorem{remark}[theorem]{Remark}

\newtheorem{question}[theorem]{Question}

\theoremstyle{theorem}

\newcommand{\N}{\mathbb{N}}
\newcommand{\Z}{\mathbb{Z}}
\newcommand{\R}{\mathbb{R}}
\newcommand{\C}{\mathbb{C}}

\usepackage{hyperref}

\usepackage[margin=25truemm]{geometry}

\title[Strong corks derived from the Akbulut cork]{Strong corks derived from the Akbulut cork}
 
\author[Tateaki Mukohara]{Tateaki Mukohara}

\address{Department of Pure and Applied Mathematics, Graduate School of Information Science and Technology, The University of Osaka, Japan}
\email{mukohara.tateaki@ist.osaka-u.ac.jp}
\begin{document}

\begin{abstract}
  We prove that the boundaries of the corks introduced by Auckly, Kim, Melvin, and Ruberman and by Tange are strong corks. Furthermore, we prove that any nontrivial linear combination of them yields a strong cork, and we construct a larger family of strong corks that generalizes them. These results rely on the instanton-theoretic invariant \(r_s\) introduced by Alfieri, Dai, Mallick, and Taniguchi.
\end{abstract}

\maketitle

\section{Introduction}
In \cite{Akbulut1991AFC}, Akbulut provided the first example of a \(cork\), which is defined as a compact, contractible 4-manifold \(W\) equipped with an orientation-preserving diffeomorphism \(\tau\) on its boundary that does not extend to a diffeomorphism of \(W\). (Some authors require \(W\) to be Stein, such as in \cite{zbMATH05660539}.) By Freedman's theorem \cite{Freedman1982TheTO}, \(\tau\) extends over \(W\) as a homeomorphism. Akbulut proved that the smooth structure of \(K3\# \overline{\C P^2}\) can be changed by cutting out an embedded copy of \(W\) and regluing it via \(\tau\). Such an operation is called a \(cork\ twist\).
More generally, Matveyev \cite{Matveyev1995ADO} and Curtis-Freedman-Hsiang-Stong \cite{Curtis1996ADT} showed that any two exotic smooth structures on a closed, simply-connected 4-manifold are related by a cork twist. Numerous examples of corks have been constructed; see, for example, \cite{zbMATH05660539} and \cite{Gompf2016InfiniteOC}.
Corks are typically detected by embedding them into a closed 4-manifold with a non-vanishing smooth 4-manifold invariant, or by utilizing Stein structures and the adjunction inequality \cite{zbMATH01057263}.

Recently, Lin, Ruberman, and Saveliev introduced the notion of a \(strong\ cork\) in \cite{Lin2018OnTF}. A strong cork is defined as a homology 3-sphere \(Y\) which bounds at least one compact, contractible 4-manifold, equipped with an orientation-preserving involution \(\tau\) that does not extend as a diffeomorphism over \(any\) smooth homology 4-ball bounded by \(Y\). In \cite{Lin2018OnTF}, they proved that the boundary of the Akbulut cork is a strong cork using monopole Floer homology. Subsequently, in \cite{Dai2020CorksIA}, Dai, Hedden, and Mallick exhibited many other examples of strong corks via Heegaard Floer theory. Some of them were later used to construct new closed exotic 4-manifolds in \cite{arXiv:2307.08130}. 
Further methods for detecting strong corks have also been developed using family Seiberg-Witten theory \cite{Konno2023FromDT} and instanton theory \cite{arXiv:2309.02309}.

Given a cork, we get an absolutely exotic pair of 4-manifolds as a consequence of the work of Akbulut-Ruberman \cite{Akbulut2014AbsolutelyEC}. 
In particular, given a strong cork \((Y,\tau)\), there exists a homology cobordism \(W\) from \(Y\) to another homology 3-sphere \(Y'\) such that, for \(any\) compact contractible 4-manifold \(X\) bounded by \(Y\), \(X\cup_{\textrm{id}}W\) and \(X\cup_{\tau}W\) form an exotic pair. Note that \(X\cup_{\textrm{id}}W\) is homotopy equivalent to \(X\); see also \cite{arXiv:2501.18584}.
This is one of the reasons why strong corks are important in the study of exotic 4-manifolds. 
In \cite{Kang2022OneSI}, Kang constructed a cork \((C,\tau)\) such that \(\tau\) does not extend over \(C\#(S^2\times S^2)\) as a diffeomorphism. Combined with the consequence in \cite{Akbulut2014AbsolutelyEC}, it is revealed that there exists an absolutely exotic pair that remains absolutely exotic after one stabilization. See also \cite{Guth2024InvariantSP}. In addition, (strong) corks have applications in the study of exotic disks and exotic surfaces; see, for example, \cite{zbMATH07615809, Hay21, Dai2022EquivariantKA, hayden2023stabilizationclosedknottedsurfaces, arXiv:2409.07287}.

In light of the connections with exotic phenomena mentioned above, it is important to determine whether a cork is strong. We pose the following question:

\begin{question}
    Let \((C,\tau)\) be a cork. Is the boundary \((\partial C,\tau)\) a strong cork?
\end{question}

It has been shown that many corks, including Akbulut-Yasui corks \((W_n,\tau)\) in \cite{zbMATH05660539} and the first member \((\overline{W}_1,\tau)\) of the positron corks, are strong. On the other hand, examples of non-strong corks were also constructed in \cite{arXiv:2005.08928}. 
The following are examples of known order-two corks for which it has not been explicitly stated whether they are strong:
\begin{enumerate}
    \item The positron cork \((\overline{W}_n,\tau)\) for \(n\geq2\). This family was introduced by Akbulut-Matveyev in \cite[Figure 2]{Akbulut2000ACD}, where \((\overline{W}_1,\tau)\) is the first positron mentioned above. See also \cite[Figure 1]{zbMATH05660539}.
    \item The corks \((C_n,\tau)\) and \((\overline{C}_n,\tau)\), displayed in Figure \ref{C_n}, for even integers \(n\geq 2\). These families were introduced by Auckly-Kim-Melvin-Ruberman in \cite[Figure 8]{Auckly2014StableII}. 
    (For \cite{Auckly2014StableII}, see also \cite{zbMATH06665051}.)
    Note that \((C_1,\tau)\) corresponds to the Akbulut cork \((W_1,\tau)\), and \((\overline{C}_1,\tau)\) corresponds to the first positron \((\overline{W}_1,\tau)\). For odd \(n\), it is shown in \cite{Dai2020CorksIA} that \((\partial C_n,\tau)\) and \((\partial \overline{C}_n,\tau)\) are strong corks.
    The family of corks \((C_n,\tau)\) was used to produce infinitely many exotic pairs of 4-manifolds in \cite{arXiv:1505.02551}. 
    \item The corks \((C(m),\tau)\), displayed in Figure \ref{C(m)}, for \(m\geq 2\). This family was introduced by Tange in \cite[Figure 1]{Tange2016FiniteOC}. Note that \((C(1),\tau)\) corresponds to the Akbulut cork. This family of corks \((C(m),\tau)\) was used to produce many examples of finite order corks in \cite{Tange2016FiniteOC}.
\end{enumerate}

\begin{figure}[h!]
\begin{center}
\includegraphics[width=5.3in]{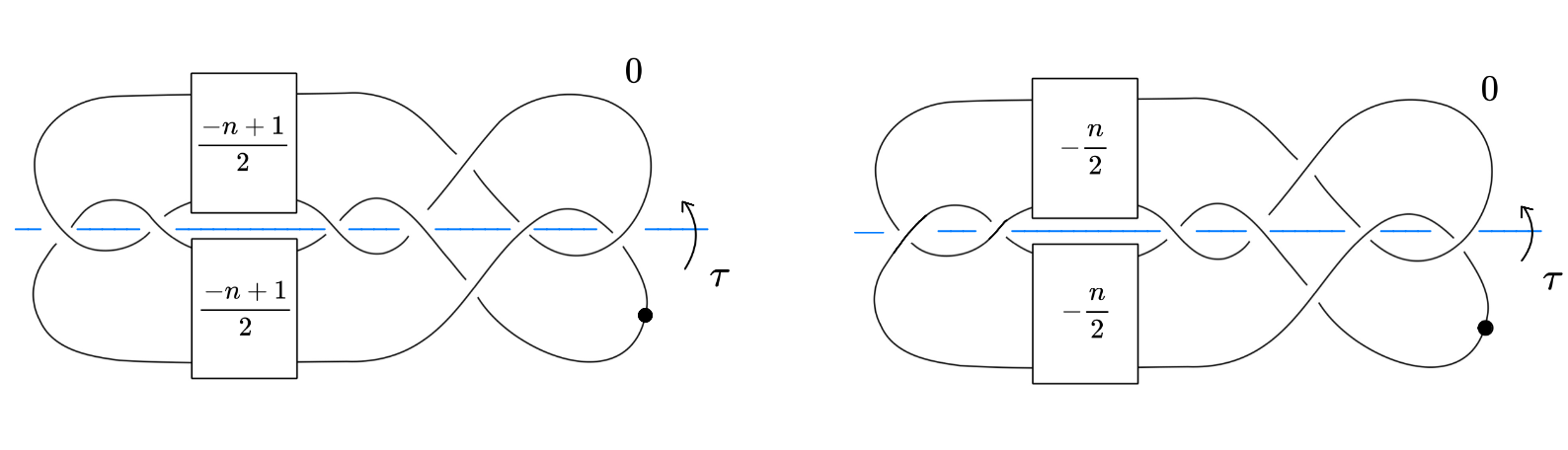}
\caption{Left: \((C_n,\tau)\). Here, the box labeled \(\dfrac{-n+1}{2}\) represents \(n-1\) negative half twists. Right: \((\overline{C}_n,\tau)\).}
\label{C_n}
\end{center}
\end{figure}

\begin{figure}[h!]
\begin{center}
\includegraphics[width=2.7in]{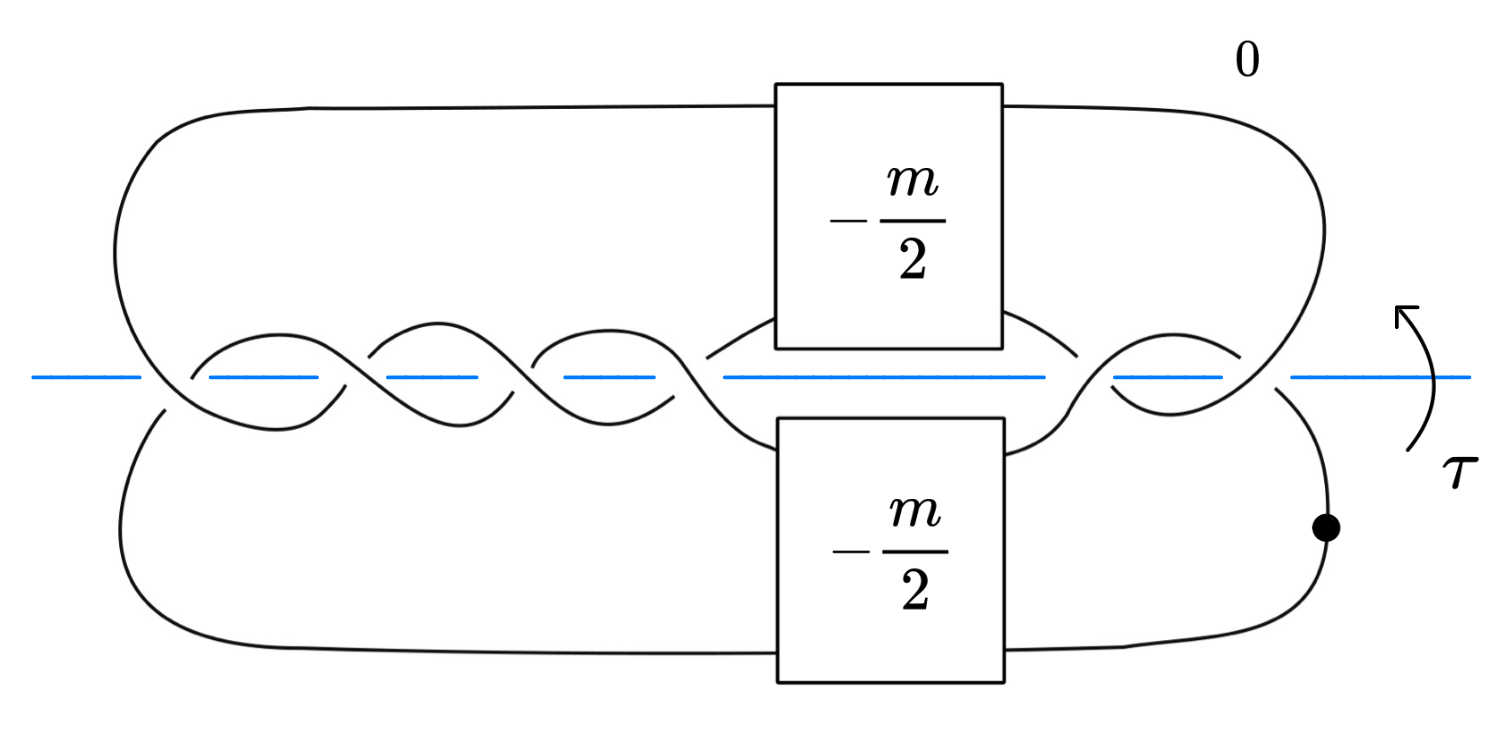}
\caption{\((C(m),\tau)\).}
\label{C(m)}
\end{center}
\end{figure}

To the best of the author's knowledge, it is unknown whether the boundaries of these corks are strong.  In this paper, we address this question for specific families of corks. The families of strong corks constructed in this paper are derived from the Akbulut cork; the first member of each family is given by the boundary of the Akbulut cork.

\subsection{Main results}

First, we construct an infinite family of strong corks that generalizes the aforementioned examples of corks \((C_n,\tau)\) and \((C(m),\tau)\):
\begin{theorem}\label{thm_Z_m,n}
    For any \(m,n\in\N\), let \(Z_{m,n}\) be the 3-manifold obtained by surgery on the two-component link displayed in Figure \ref{Z_m,n}. Equip \(Z_{m,n}\) with the indicated involution \(\tau\). Then the following statements hold:
    \begin{itemize}
        \item If \(m\leq m', n\leq n'\) and \((m,n)\neq (m',n')\), then there exists a simply-connected, equivariant, negative-definite cobordism from \((Z_{m',n'},\tau)\) to \((Z_{m,n},\tau)\).
        \item For any \(m,n\in\N\), \(r_0(Z_{m,n},\tau)<\infty\) and thus \((Z_{m,n},\tau)\) is a strong cork. 
    \end{itemize}
    Furthermore, any nontrivial linear combination of elements in either \(\{(Z_{m,n},\tau)\}_{m\in\N}\) (fixing \(n\)) or \(\{(Z_{m,n},\tau)\}_{n\in\N}\) (fixing \(m\)) yields a strong cork. Specifically, for any sequence of integers \((a_1,a_2,\cdots ,a_k)\neq (0,0,\cdots, 0)\), the equivariant connected sums
    \[
    (a_1 Z_{1,n}\# a_2 Z_{2,n}\#\cdots \# a_k Z_{k,n}, \tau)\ and\ \ (a_1 Z_{m,1}\# a_2 Z_{m,2}\#\cdots \# a_k Z_{m,k}, \tau)
    \]
    are strong corks. Here, for \(a<0\), the notation \(aZ\) denotes \(|a|(-Z)\).
\end{theorem}

\begin{figure}[h!]
\begin{center}
\includegraphics[width=3.2in]{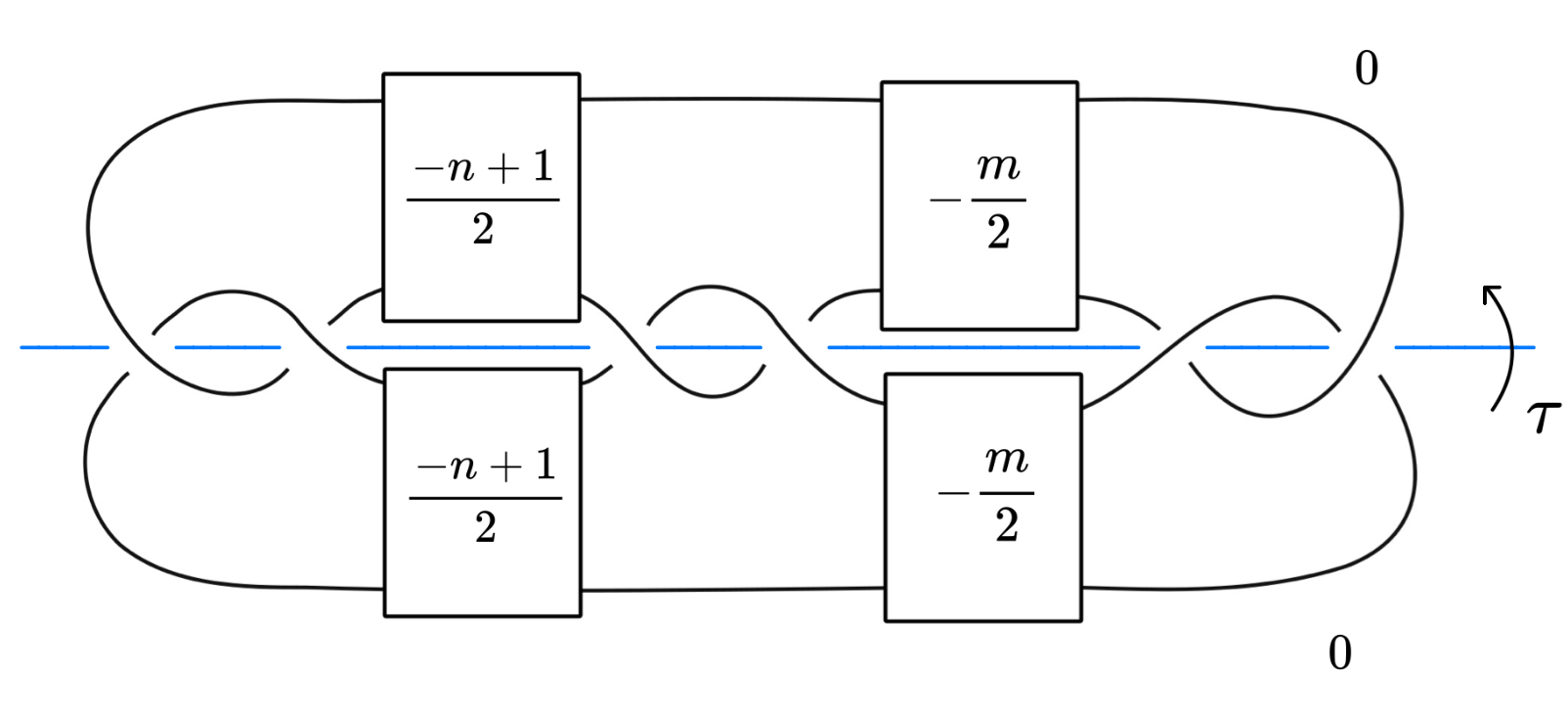}
\caption{\((Z_{m,n},\tau)\)}
\label{Z_m,n}
\end{center}
\end{figure}

For \(m=1\) and odd \(n\), it is shown in \cite[Theorem 1.12]{Dai2020CorksIA} and \cite[Theorem 1.5]{arXiv:2309.02309} that \((Z_{1,n},\tau)\) is a strong cork.
However, linear combinations of these corks are not discussed in these papers.

When \(m=1\), \((Z_{1,n},\tau)\) coincides with the boundary \((\partial C_n,\tau)\) of Auckly-Kim-Melvin-Ruberman's cork in \cite{Auckly2014StableII}, and when \(n=1\), \((Z_{m,1},\tau)\) coincides with the boundary \((\partial C(m),\tau)\) of Tange's cork in \cite{Tange2016FiniteOC}. In particular, for \(m=n=1\), \((Z_{1,1},\tau)\) corresponds to the boundary of the Akbulut cork.

The proof of Theorem \ref{thm_Z_m,n} relies on the instanton-theoretic invariant \(r_s(Y,\tau)\) introduced in \cite{arXiv:2309.02309}. 

\begin{remark}
    \((Z_{m,n}\# (-Z_{m,n}),\tau)\) is a boundary of a cork, but it is not a strong cork; that is, \(\tau\) extends over some homology ball that \(Z_{m,n}\# (-Z_{m,n})\) bounds. However, \(\tau\) does not extend over any contractible 4-manifold that \(Z_{m,n}\# (-Z_{m,n})\) bounds. See Remark \ref{rem}.
\end{remark}

On the other hand, using the Heegaard-Floer-theoretic invariants \(h_{\tau}(Y)\) and \(h_{\iota\circ\tau}(Y)\) introduced in \cite{Dai2020CorksIA}, we construct the family of strong corks in the following theorem:

\begin{theorem}\label{thm_Y_m,n}
    For \(m,n\in\N\), let \(Y_{m,n}\) be the 3-manifold obtained by surgery on the two-component link displayed in Figure \ref{Y_m,n}. Equip \(Y_{m,n}\) with the indicated involution \(\sigma\). Then the following statements hold for any \(m,n\in \N\):
    \begin{itemize}
        \item There exists an equivariant negative-definite cobordism from \((Y_{m,n+1},\sigma)\) to \((Y_{m,n},\sigma)\).
        \item There exists an equivariant negative-definite cobordism from \((Y_{m,n},\sigma)\) to \((\Sigma(2,2m+1,4m+3),\sigma)\), where \(\sigma\) is the involution displayed in Figure \ref{torus}.
        \item  \(h_{\iota\circ\sigma}(Y_{m,n})<0\) and thus \((Y_{m,n},\sigma)\) is a strong cork.
    \end{itemize}
    Moreover, it is also a strong cork by introducing any number of symmetric pairs of negative full twists.
\end{theorem}

\begin{figure}[h!]
\begin{center}
\includegraphics[width=3.3in]{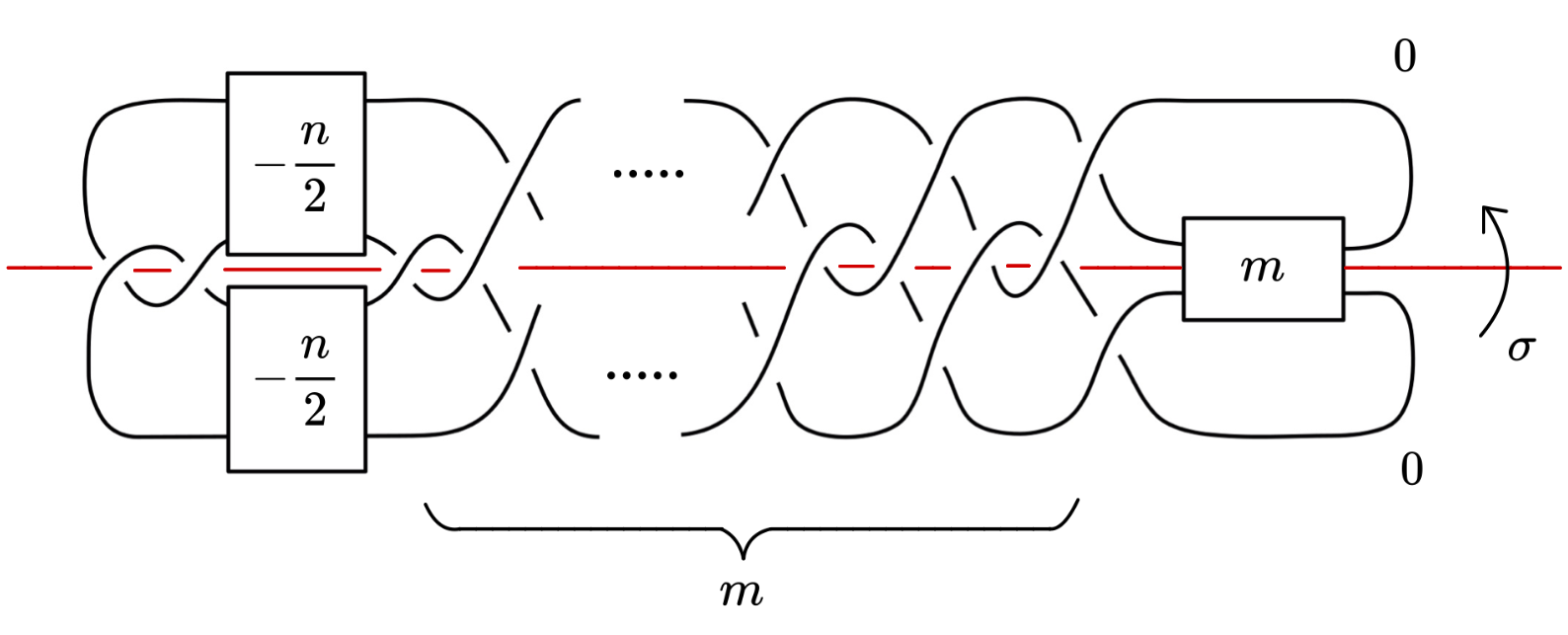}
\caption{\((Y_{m,n},\sigma)\).}
\label{Y_m,n}
\end{center}
\end{figure}

The proof of Theorem \ref{thm_Y_m,n} relies on the non-triviality of \(h_{\iota\circ\sigma}(\Sigma (2,2m+1,4m+3))\), where \(\Sigma(2,2m+1,4m+3)\) is a Brieskorn homology sphere equipped with the involution \(\sigma\) displayed in Figure \ref{torus}.

When \(n=1\), \(Y_{m,1}\) is diffeomorphic to the boundary \(\partial W_m\) of the Akbulut-Yasui cork in \cite[Figure 1]{zbMATH05660539}.
For the case \(m=1\), \((Y_{1,n},\sigma)\) is shown in Figure \ref{Y_1,n}. 
The following theorem can be proved analogously to Theorem \ref{thm_Z_m,n}:

\begin{theorem}\label{thm_Y_1,n}
    Any nontrivial linear combination of elements in \(\{(Y_{1,n},\sigma)\}_{n\in\N}\) is a strong cork. 
\end{theorem}

\begin{figure}[h!]
\begin{center}
\includegraphics[width=2.3in]{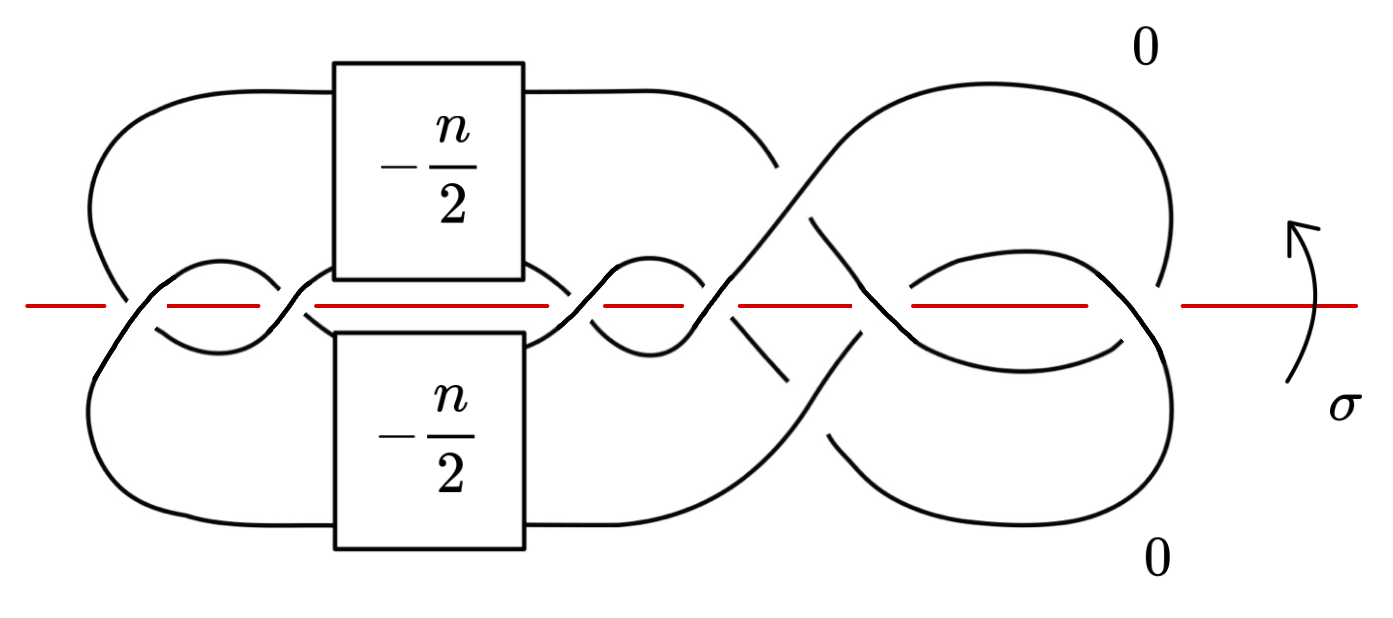}
\caption{\((Y_{1,n},\sigma)\).}
\label{Y_1,n}
\end{center}
\end{figure}

As shown in \cite[Theorem 4.2]{Auckly2014StableII}, \((Z_{1,n},\tau)\), appearing in Theorem \ref{thm_Z_m,n}, is equivariantly diffeomorphic to \((S^3_{+1}(K_n),\tau)\), where \((K_n,\tau)\) is the strongly invertible slice knot displayed on the left side of Figure \ref{K_n}. Using a technique similar to that in \cite[Theorem 4.2]{Auckly2014StableII}, one can show that \((Y_{1,n},\sigma)\) is equivariantly diffeomorphic to \((S^3_{+1}(K_n),\sigma)\), where \((K_n,\sigma)\) is  displayed on the right side of Figure \ref{K_n}. 
The following theorem is an extension of \cite[Theorem 1.5]{arXiv:2309.02309}:

\begin{theorem}\label{thm_K_n}
    For \(n\in\N\), let \(K_n\) be the strongly invertible slice knot displayed in Figure \ref{K_n}, equipped with the indicated involutions \(\tau\) and \(\sigma\). Then any nontrivial linear combination of elements in \(\{(S^3_{1/m}(K_n),\tau)\}_{m\in\N}\) is a strong cork. This statement holds with \(\tau\) replaced by \(\sigma\).
\end{theorem}

\begin{figure}[h!]
\begin{center}
\includegraphics[width=5.3in]{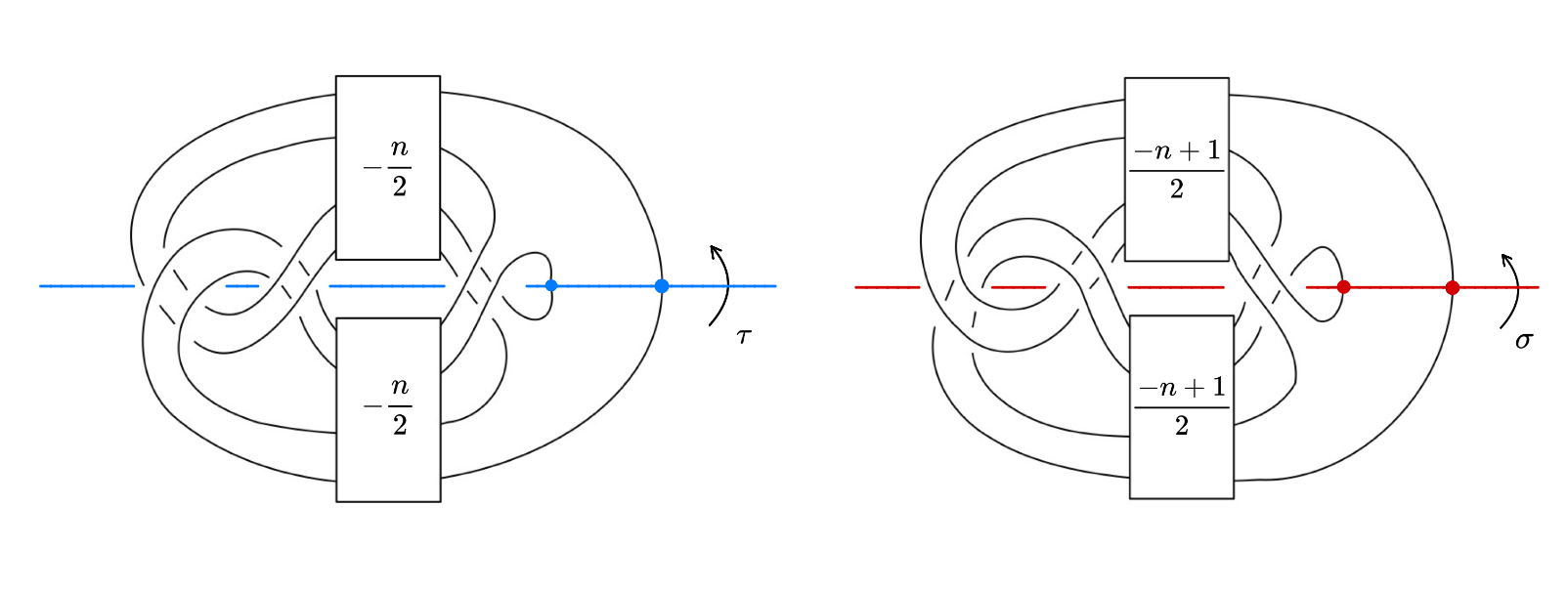}
\caption{Left: \((K_n,\tau)\). Right: \((K_n,\sigma)\).}
\label{K_n}
\end{center}
\end{figure}

When \(n=1\), the first member of the family \(K_n\) is \(K_1=\overline{9}_{46}\), and \((S^3_{+1}(K_1),\tau)\) is the boundary of the Akbulut cork. In the case where \(n\) is odd and the involution is \(\tau\), Theorem \ref{thm_K_n} was established in \cite{arXiv:2309.02309}.

The following proposition implies that the involutions of \((Z_{1,n},\tau)\) and \((Y_{1,n},\sigma)\) are induced by distinct symmetries of \(K_n\), respectively, although \(Z_{1,n}\) and \(Y_{1,n}\) are diffeomorphic.

\begin{proposition}\label{prop_sakuma}
    For any \(n\in\N\), the knot \(K_n\), displayed in Figure \ref{K_n}, admits two strong involutions \(\tau\) and \(\sigma\) such that the pairs \((K_n,\tau)\) and \((K_n,\sigma)\) are not Sakuma equivalent.
\end{proposition}

Using SnapPy \cite{SnapPy}, we verified that, for \(1\leq n\leq 5\), \(K_n\) is a hyperbolic knot and the isometry group of \(S^3\setminus K_n\) is isomorphic to \(\Z_2\oplus \Z_2\). 
By Mostow rigidity theorem \cite{LIHS1968QuasiconformalMI}, if \(K\) is a hyperbolic knot, then the symmetry group of \(K\) is isomorphic to the isometry group of \(S^3\setminus K\).
The knot \(K_n\) admits distinct strong involutions \(\tau\) and \(\sigma\). Since the product of two distinct strong involutions is orientation-preserving, the third non-trivial element of the group must be a periodic involution. 
Consequently, \(K_n\) admits exactly two strong involutions for \(1\leq n\leq 5\).

\begin{question}
    For any \(n\geq 1\), is the knot \(K_n\) hyperbolic? When \(K_n\) is hyperbolic, is the isometry group of \(S^3\setminus K_n\) isomorphic to \(\Z_2\oplus \Z_2\)?
\end{question}

We prove Theorems \ref{thm_Z_m,n}, \ref{thm_Y_m,n}, \ref{thm_Y_1,n}, and \ref{thm_K_n} in Section \ref{sectioncork}, and we prove Proposition \ref{prop_sakuma} in Section \ref{appendix} (appendix).

We close this section with the following question.

\begin{question}
    Are the boundaries of the positron corks \((\overline{W}_n,\tau)\) for \(n\geq 2\) and Auckly-Kim-Melvin-Ruberman's corks \((\overline{C}_n,\tau)\) for even \(n\geq 2\) (which are not covered in this paper) strong corks?
\end{question}

\subsection*{Acknowledgements}
The author would like to express his sincere gratitude to Masaki Taniguchi for his helpful advice, insightful discussions, and encouragement. He would also like to thank Kouichi Yasui for his helpful comments on the draft and constructive discussions, and Irving Dai, Abhishek Mallick, and Motoo Tange for their valuable comments on the draft.

\section{Definitions and background}
\subsection{Equivariant handle attachment} 
In this section, we briefly review the basics of equivariant handle attachment.
First, we begin with the definition of an involutive link and an equivariant knot.

\begin{definition}
    Let \(L\) be a link in a 3-manifold \(Y\), and let \(\tau\) be an orientation-preserving involution on \(Y\). 
    We say that \((L,\tau)\) is an \(involutive\ link\)
    if \(\tau\) fixes \(L\) setwise. 
    In particular, in the case  that \(L\) is a knot, 
    we say that \((L,\tau)\) is an \(equivariant\ knot\).
    Let \((K,\tau)\) be an equivariant knot.
    If \(\tau\) fixes two points on \(K\), we say that \((K,\tau)\) is \(strongly\ invertible\). If \(\tau\) has no fixed points on \(K\), we say that \((K,\tau)\) is \(periodic\).
\end{definition}

\begin{definition}
    Two strongly invertible knots \((K,\tau)\) and \((K',\tau')\) are said to be \(Sakuma\ equivalent\) if there exists an orientation-preserving diffeomorphism \(f:S^3\rightarrow S^3\) such that \(f(K)=K'\) and \(f\circ\tau=\tau'\circ f\).
\end{definition}

By the work of Waldhausen \cite{Waldhausen1969berID}, any orientation-preserving involution on \(S^3\) is conjugate to a rotation about an unknot.
In the case that \(Y=S^3\), we will identify \(S^3\) with the one-point compactification of \(\R ^3\) and often depict \(\tau\) as a \(180^{\circ}\) rotation around an axis in \(\R ^3\).

If two strongly invertible knots \((K,\tau)\) and \((K',\tau')\) are Sakuma equivalent, then they are related by a sequence of the involutive Reidemeister moves, together with an equivariant isotopy passing through the point at \(\infty\) and a \(180^{\circ}\) rotation about the origin; see \cite[Theorem 2.4]{Lobb2019ARO} and \cite[Theorem 2.10]{arXiv:2507.13642}.

In \cite[Section 5]{Dai2020CorksIA}, it is shown that if \(K\) is an equivariant knot in \(Y\) with symmetry \(\tau\), then \(\tau\) extends to an involution on any Dehn surgery along \(K\). We also denote it by \(\tau\). Moreover, \(\tau\) extends over the 2-handle cobordism given by attaching a 2-handle along \(K\).
Similarly, if \(L\) is a link equipped with \(\tau\) that exchanges some pairs of components with the same framings, then \(\tau\) extends to involutions on surgered manifolds and 2-handle cobordisms.

We define (\(\Z_2\)-)equivariant handles and equivariant attaching maps. For a more general setting, see \cite{Was69} and \cite{ms2025}.

\begin{definition}
    Let \(h=D^k\times D^{n-k}\) be a \(k\)-handle equipped with an orientation-preserving involution \(\tau\).
    We say that \((h,\tau)\) is an \(equivariant\ handle\) if 
    \[
    \tau(x,y)=(\alpha(x),\beta(y)),
    \]
    where \(\alpha:D^k\rightarrow D^k,\ \beta: D^{n-k}\rightarrow D^{n-k}\) are involutions smoothly conjugate to \(A|_{D^k}\) and \(B|_{D^{n-k}}\) for some \(A\in O(k),\ B\in O(n-k)\), respectively.
    
    Let \(h_1\sqcup h_2\) be a disjoint union of \(k\)-handles \(D_i^k\times D_i^{n-k}\), and \(\tau\) be an orientation-preserving involution on \(h_1\sqcup h_2\) that exchanges the two components.
    We say that \((h_1\sqcup h_2,\tau)\) is an \(equivariant\ pair\ of\ handles\) if 
    \[
    \tau(x,y)=(\alpha(x),\beta(y)),
    \]
    where \(\alpha:D_1^k\sqcup D_2^k\rightarrow D_1^k\sqcup D_2^k,\ \beta: D_1^{n-k}\sqcup D_2^{n-k}\rightarrow D_1^{n-k}\sqcup D_2^{n-k}\) are involutions satisfying
    \[
    \alpha(D^k_1)=D_2^k,\ \beta(D_1^{n-k})=D_2^{n-k},
    \]
    and (after identifying the disk factors \(D_i^k, D_i^{n-k}\) with the standard disks \(D^k, D^{n-k}\)), \(\alpha|_{D_i^k}\) and \(\beta|_{D_i^{n-k}}\) are smoothly conjugate to \(A|_{D^k}\) and \(B|_{D^{n-k}}\) for some \(A\in O(k),\ B\in O(n-k)\), respectively.
\end{definition}

\begin{definition}
    Let \((X,\tau)\) be an \(n\)-manifold with boundary, equipped with an involution.
    \begin{itemize}
        \item Let \((h,\tau_h)\) be an equivariant handle and \(\varphi:\partial D^k\times D^{n-k}\rightarrow \partial X\) be an attaching map. We say that \(\varphi\) is \(equivariant\) if 
        \[
        \tau|_{\partial X}\circ\varphi=\varphi\circ\tau_h|_{\partial D^k\times D^{n-k}}.
        \]
        \item Let \((h_1\sqcup h_2,\tau_h)\) be an equivariant pair of handles and \(\varphi_i:\partial D_i^k\times D_i^{n-k}\rightarrow \partial X\) be an attaching map for each \(h_i\) (\(i=1,2\)). We say that \(\varphi_1\) and \(\varphi_2\) are \(equivariant\) if 
        \[
        \tau|_{\partial X}\circ\varphi_1=\varphi_2\circ\tau_h|_{\partial D_1^k\times D_1^{n-k}}.
        \]
    \end{itemize}
    
\end{definition}

\begin{definition}
    Let \((X,\tau)\) and \((X',\tau')\) be manifolds equipped with involutions. 
    We say that \((X,\tau)\) and \((X',\tau')\) are \(equivariantly\ diffeomorphic\) if there exists an orientation-preserving diffeomorphism \(f:X\rightarrow X'\) such that \(f\circ\tau=\tau'\circ f\).
\end{definition}

Now, we focus on the 4-dimensional case and consider the equivariant attachment of equivariant 2-handles and equivariant pairs of 2-handles to \((X,\tau)\), where \(X\) is a 4-manifold with boundary and \(\tau\) is an orientation-preserving involution.
It follows from the definition that 
the union of attaching circles of equivariant (pairs of) 2-handles forms an involutive link in \(\partial X\).
In particular, attaching circles corresponding to an equivariant pair of 2-handles are exchanged by the involution \(\tau|_{\partial X}\) and they are assigned the same framing coefficients.

Given an \(n\)-framed equivariant knot \((K,\tau)\) in \(\partial X\), there exists an equivariant 2-handle \((h,\tau_h)\) and an equivariant attaching map \(\varphi\) whose attaching sphere is \(K\). 
(The involution \(\tau_h\) is explicitly described in \cite[Lemma 5.3]{Dai2020CorksIA}.)
Then 
\(\tau\) and \(\tau_h\) glue together to define an involution \(\tilde{\tau}\) on the resulting 4-manifold  \(X\cup_{\varphi} h\). 
By the equivariant isotopy extension theorem (see, for example, \cite[Theorem 3.6.1]{Fie07}),  the equivariant diffeomorphism type of \((X\cup_{\varphi} h,\tilde{\tau})\) depends only on the equivariant isotopy type of the attaching map \(\varphi\).
Indeed, if the two equivariant attaching maps \(\varphi_0\) and \(\varphi_1\) are equivariantly isotopic,
then the equivariant isotopy extension theorem gives an equivariant ambient isotopy 
\[
H_t: \partial X\rightarrow\partial X\ (t\in [0,1])
\]
such that \(H_1\circ \varphi_0=\varphi_1\). Choose an equivariant collar neighborhood of \(\partial X\). 
Extend \(H_t\) over this collar, set the extension equal to the identity outside the collar, and glue with the identity on the 2-handle \(h\).
This gives
the equivariant diffeomorphism between the resulting 4-manifolds \((X\cup_{\varphi_0} h,\tilde{\tau})\) and \((X\cup_{\varphi_1} h,\tilde{\tau})\).
The same argument applies to an equivariant pair of 2-handles.

\subsection{Strong corks}
We recall the definitions of a cork and a strong cork.
\begin{definition}
    Let \(C\) be a compact, contractible 4-manifold and \(\tau : \partial C\rightarrow\partial C\) be an orientation-preserving diffeomorphism. We say that \((C,\tau)\) is a \(cork\) if \(\tau\) does not extend over \(C\) as a diffeomorphism.
\end{definition}

In this paper, we do not require \(C\) to be Stein, although this condition is sometimes required, such as in \cite{zbMATH05660539}.

\begin{definition}
    Let \(Y\) be an oriented integer homology 3-sphere and \(\tau: Y\rightarrow Y\) be an orientation-preserving diffeomorphism. Suppose that \(Y\) bounds at least one compact, contractible 4-manifold. We say that \((Y,\tau)\) is a \(\mathit{strong\ cork}\) if \(\tau\) does not extend as a diffeomorphism over any smooth homology ball that \(Y\) bounds. 
\end{definition}

In this paper, we consider only the case where \(\tau\) is an involution.

\begin{definition}
    Let \((Y,\tau)\) and \((Y',\tau')\) be oriented integer homology 3-spheres equipped with orientation-preserving involutions. We say that \((W,\tilde{\tau})\) is an \(equivariant\ cobordism\) from \((Y,\tau)\) to \((Y',\tau')\) if \(W\) is a cobordism from \(Y\) to \(Y'\) and \(\tilde{\tau}\) is a diffeomorphism of \(W\) such that \(\tilde{\tau}|_{Y}=\tau\) and \(\tilde{\tau}|_{Y'}=\tau'\).
\end{definition}

If \((Y_1,\tau_1)\) and \((Y_2,\tau_2)\) are oriented homology 3-spheres equipped with orientation-preserving involutions, and the fixed-point set of \(\tau_i\) is diffeomorphic to \(S^1\), then we can define their equivariant connected sum \((Y_1\# Y_2, \tau_1\#\tau_2)\) as follows. 
Let \(p_1\in Y_1\) and \(p_2\in Y_2\) be fixed points of \(\tau_1\) and \(\tau_2\), respectively. 
Let \(f_i: D^3\rightarrow Y_i\) be an embedding satisfying the following conditions:
\begin{enumerate}
    \item \(f_i(0)=p_i\) and \(\tau_i(f_i(D^3))=f_i(D^3)\).
    \item \(f^{-1}_i\circ\tau_i\circ f_i(x,y,z)=(x,-y,-z)\).
    \item The orientation of \(\textrm{Fix}(\tau_i)\) coincides with the orientation of \(f_i(D^3)\cap \textrm{Fix}(\tau_i)\) induced from the natural orientation of \(\{(x,0,0)\in D^3)\}\).
\end{enumerate}
We set \(D^3_{1/2}=\{(x,y,z)\in D^3\mid x^2+y^2+z^2\leq 1/2\}\) and define \(\phi:\partial D^3_{1/2}\rightarrow \partial D^3_{1/2}\) by \((x,y,z)\mapsto (x,-y,z)\).
Then the connected sum is defined as
\(
Y_1\# Y_2=(Y_1\setminus f_1(\textrm{Int}(D^3_{1/2})))\cup_{\phi} (Y_2\setminus f_2(\textrm{Int}(D^3_{1/2})))
\)
and we can define the involution \(\tau_1\#\tau_2\) on \(Y_1\# Y_2\) such that \((\tau_1\#\tau_2)|_{Y_i\setminus f_i(D^3_{1/2})}=\tau_i\). This operation is independent of the choice of \(p_1\) and \(p_2\).

\section{Detecting strong corks}\label{sectioncork}
\subsection{Heegaard-Floer-theoretic methods}
Many new families of strong corks have been discovered using Heegaard-Floer-theoretic methods developed in \cite{Dai2020CorksIA}. The main tool is the following theorem:
\begin{theorem}\cite[Theorem 1.1]{Dai2020CorksIA}\label{thm3.1}
    Let \(Y\) be an oriented integer homology 3-sphere and \(\tau\) be an orientation-preserving involution on \(Y\). Then there are two invariants 
    \[h_{\tau}(Y)=[(CF^-(Y)[-2],\tau)],\ h_{\iota\circ\tau}(Y)=[(CF^-(Y)[-2],\iota\circ\tau)]
    \]
    associated with the pair \((Y,\tau)\). If either \(h_{\tau}(Y)\neq 0\) or \(h_{\iota\circ\tau}(Y)\neq 0\), then \(\tau\) does not extend to a diffeomorphism of any homology ball bounded by \(Y\). Here, we denote \(0=h_{\textrm{id}}(S^3)\).
\end{theorem}

The invariants \(h_\tau(Y)\) and \(h_{\iota\circ\tau}(Y)\) take values in \(\mathfrak{I}\), the group of \(\iota\)-complexes modulo local equivalence, which was defined in \cite{Hendricks2017ACS}. We briefly review the definition of the invariants \(h_\tau(Y)\) and \(h_{\iota\circ\tau}(Y)\). In \cite[Section 4]{Dai2020CorksIA}, it is shown that \(\tau\) induces a homotopy involution \(\tau:CF^-(Y)\rightarrow CF^-(Y)\) and \((CF^-(Y),\tau)\) constitutes an \(\iota\)-complex. Taking the local equivalence class of this \(\iota\)-complex gives 
\[
h_{\tau}(Y)=[(CF^-(Y)[-2],\tau)]\in\mathfrak{I}.
\] 
Considering the map \(\iota\circ\tau\) instead of \(\tau\) gives another \(\iota\)-complex and its local equivalence class 
\[
h_{\iota\circ\tau}(Y)=[(CF^-(Y)[-2],\iota\circ\tau)]\in\mathfrak{I},
\]
where \(\iota\) is the homotopy involution on \(CF^{-}(Y)\) defined in \cite{zbMATH06736599}.
The local equivalence group \(\mathfrak{I}\) admits a partial order; see \cite[Definition 3.6]{Dai2020CorksIA}. The invariants \(h_\tau(Y)\) and \(h_{\iota\circ\tau}(Y)\) are monotonic under an appropriate equivariant negative-definite cobordism, as in the following theorem:
\begin{theorem}\cite[Theorem 1.5]{Dai2020CorksIA}\label{thm3.2}
    Let \((Y,\tau)\) and \((Y',\tau')\) be oriented integer homology 3-spheres equipped with orientation-preserving involutions \(\tau\) and \(\tau'\).
    \begin{enumerate}
        \item Let \((K,\tau)\) be an equivariant knot in \(Y\). Suppose that \(Y'\) is obtained from \(Y\) by performing \((-1)\)-surgery on \(K\), and \(\tau'\) is the extension of \(\tau\). If \((K,\tau)\) is periodic, then
        \[
        h_\tau(Y)\leq h_{\tau'}(Y').
        \]
        If \((K,\tau)\) is strongly invertible, then
        \[
        h_{\iota\circ\tau}(Y)\leq h_{\iota\circ\tau'}(Y').
        \]
        \item Let \(L=L_1\cup L_2\) be a two-component link in \(Y\) with algebraic linking number zero whose components \(L_1,\ L_2\) are interchanged by \(\tau\). 
        Suppose that \(Y'\) is obtained from \(Y\) by performing \((-1)\)-surgery on each component of \(L\), and \(\tau'\) is the extension of \(\tau\). Then
        \[
        h_{\tau}(Y)\leq h_{\tau'}(Y')\ 
        and\ 
        h_{\iota\circ\tau}(Y)\leq h_{\iota\circ\tau'}(Y').
        \]
    \end{enumerate}
\end{theorem}

Attaching a 2-handle along \(K\) (or 2-handles along \(L\)), as described in Theorem \ref{thm3.2}, yields an equivariant negative-definite cobordism from \((Y,\tau)\) to \((Y',\tau')\). In case (1) of Theorem \ref{thm3.2}, the resulting cobordism is called a \(spin^c\)-\(fixing\ cobordism\) if \(K\) is periodic, and is called a \(spin^c\)-\(conjugating\ cobordism\) if \(K\) is strongly invertible. In case (2), the corresponding cobordism is called an \(interchanging\ (-1,-1)\)-\(cobordism\).

Theorem \ref{thm3.2} enables us to show the non-triviality of \(h_{\tau}(Y)\) or \(h_{\iota\circ\tau}(Y)\) by constructing an equivariant negative-definite cobordism from \((Y,\tau)\) to another pair \((Y',\tau')\) whose invariants are already understood. The following result is useful for constructing examples of strong corks:

\begin{lemma}\cite[Lemma 7.8]{Dai2020CorksIA}\label{lem3.3}
    Let \(T_{2,2n+1}\) be the right-handed \((2,2n+1)\)-torus knot, equipped with the involutions \(\tau\) and \(\sigma\) indicated in Figure \ref{torus}. Let \(B_n=S^3_{-1}(T_{2,2n+1})=\Sigma (2,2n+1, 4n+3)\). Then
    \[
    h_{\tau}(B_n)=h_{\iota\circ\sigma}(B_n)<0.
    \]
\end{lemma}

\begin{figure}[h!]
\begin{center}
\includegraphics[width=2.0in]{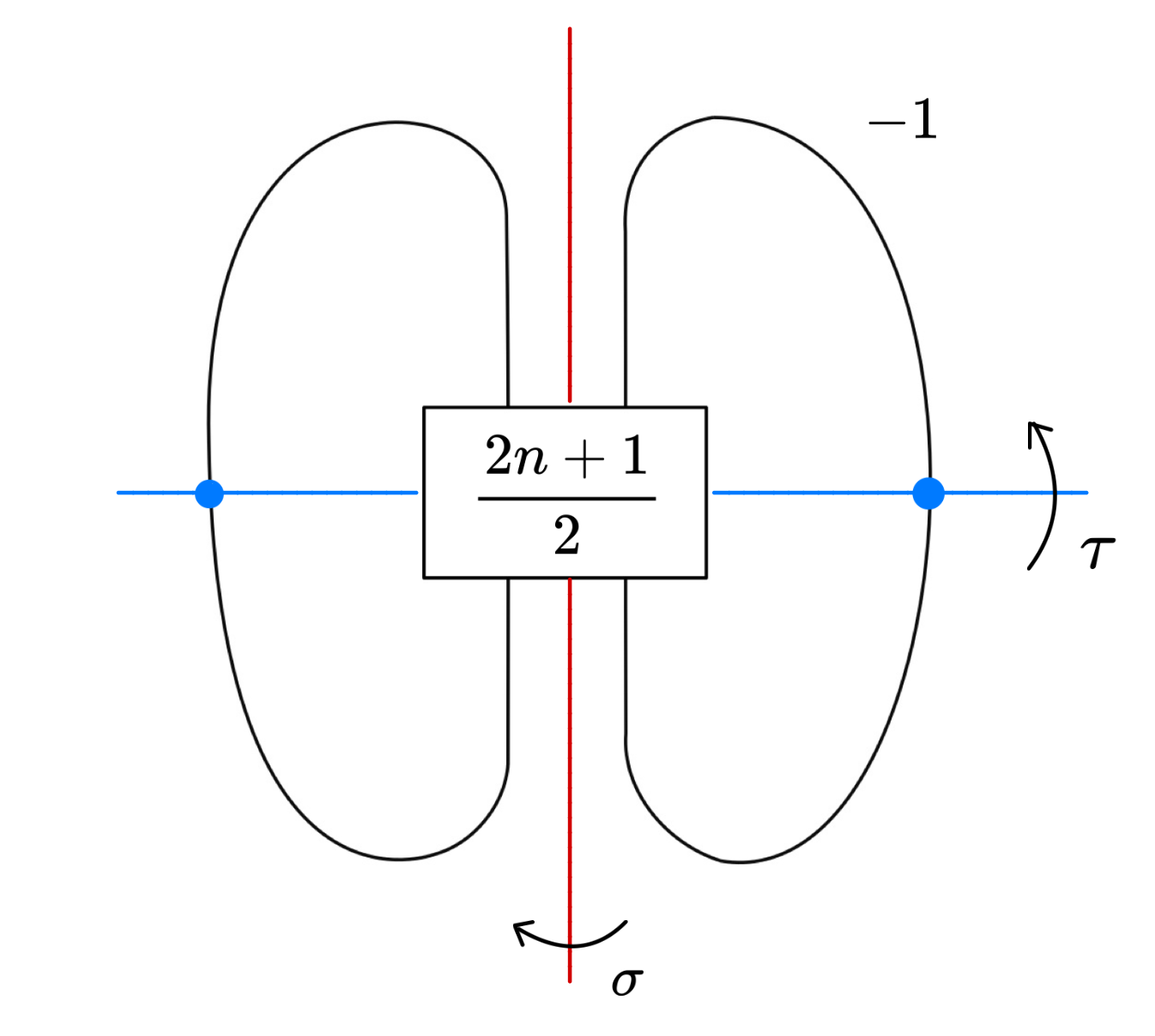}
\caption{\(S^3_{-1}(T_{2,2n+1})=\Sigma (2,2n+1, 4n+3)\)}
\label{torus}
\end{center}
\end{figure}

As shown in Figure \ref{9_46}, the knot \(\overline{9}_{46}=P(-3,3,-3)\) admits two strong involutions \(\tau\) and \(\sigma\). It is known that \((S^3_{+1}(\overline{9}_{46}),\tau)\) is the boundary of the Akbulut cork. Note that \((Y_{1,1},\sigma)\) appearing in Theorem \ref{thm_Y_m,n} is equivariantly diffeomorphic to \((S^3_{+1}(\overline{9}_{46}),\sigma)\).

\begin{theorem}\cite[Theorem 1.11]{Dai2020CorksIA}
     Let \(Y=S^3_{+1}(\overline{9}_{46})\) be given by \((+1)\)-surgery on the knot \(\overline{9}_{46}\) equipped with the indicated involutions \(\tau\) and \(\sigma\). Then
    \[
    h_{\tau}(Y)<0,\ h_{\iota\circ\sigma}(Y)<0.
    \]
\end{theorem}

\begin{figure}[h!]
\begin{center}
\includegraphics[width=1.4in]{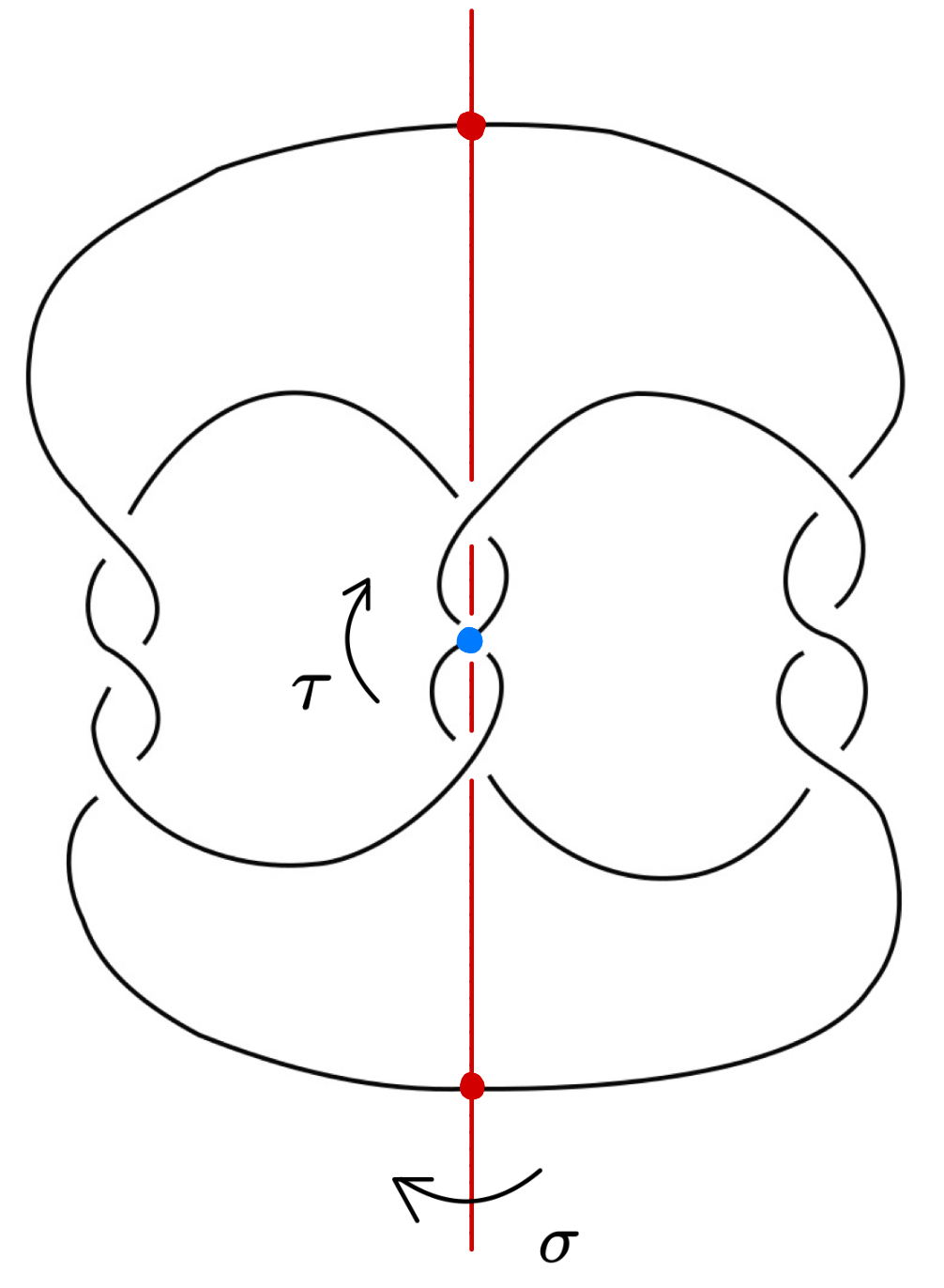}
\caption{\(\overline{9}_{46}=P(-3,3,-3)\).}
\label{9_46}
\end{center}
\end{figure}

\subsection{Proof of Theorem \ref{thm_Y_m,n}}
The following lemma provides useful equivariant operations that will be used in the constructions below. These operations were used in \cite[Theorem 4.2]{Auckly2014StableII} without proof. Here, we show them explicitly.

\begin{figure}[h!]
\begin{center}
\includegraphics[width=4.0in]{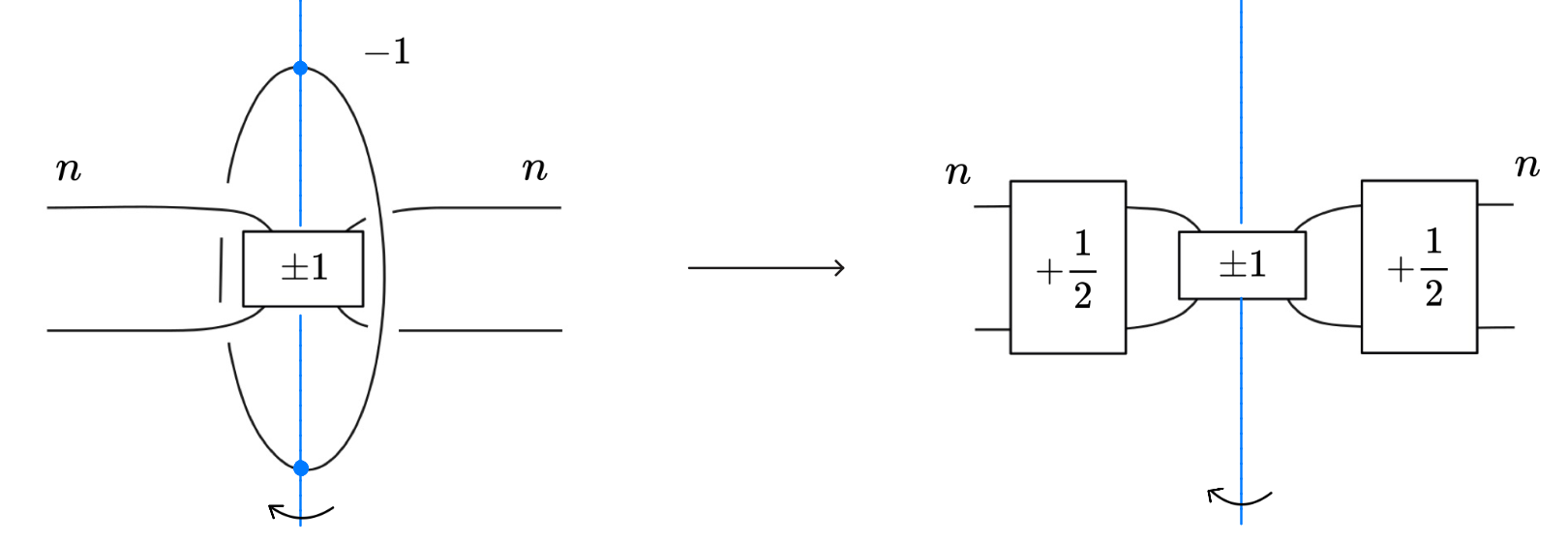}
\caption{The operation in Lemma \ref{lem3.5}, where the central unknot has framing \(-1\).}
\label{operation-1}
\end{center}
\end{figure}

\begin{figure}[h!]
\begin{center}
\includegraphics[width=4.0in]{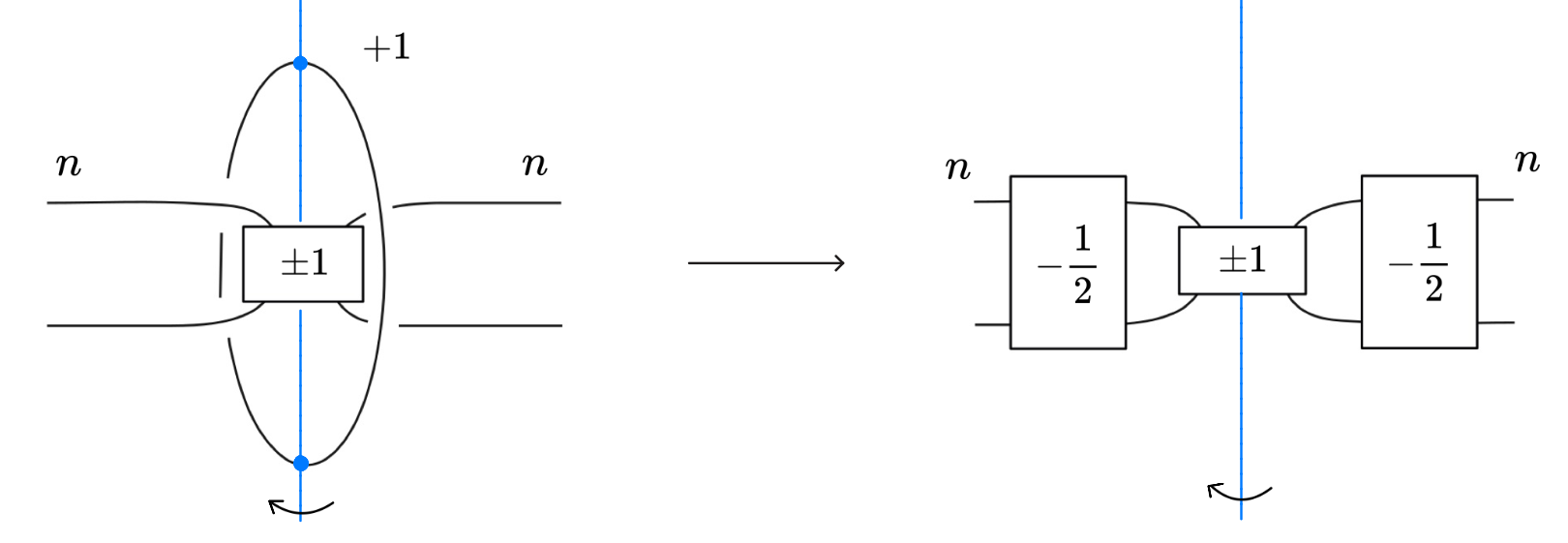}
\caption{The operation in Lemma \ref{lem3.5}, where the central unknot has framing \(+1\).}
\label{operation+1}
\end{center}
\end{figure}

\begin{lemma}\label{lem3.5}
    The operations displayed in Figures \ref{operation-1} and \ref{operation+1} can be performed equivariantly. If the left and right components are distinct, their framings remain unchanged.
\end{lemma}

\begin{figure}[h!]
\begin{center}
\includegraphics[width=4.6in]{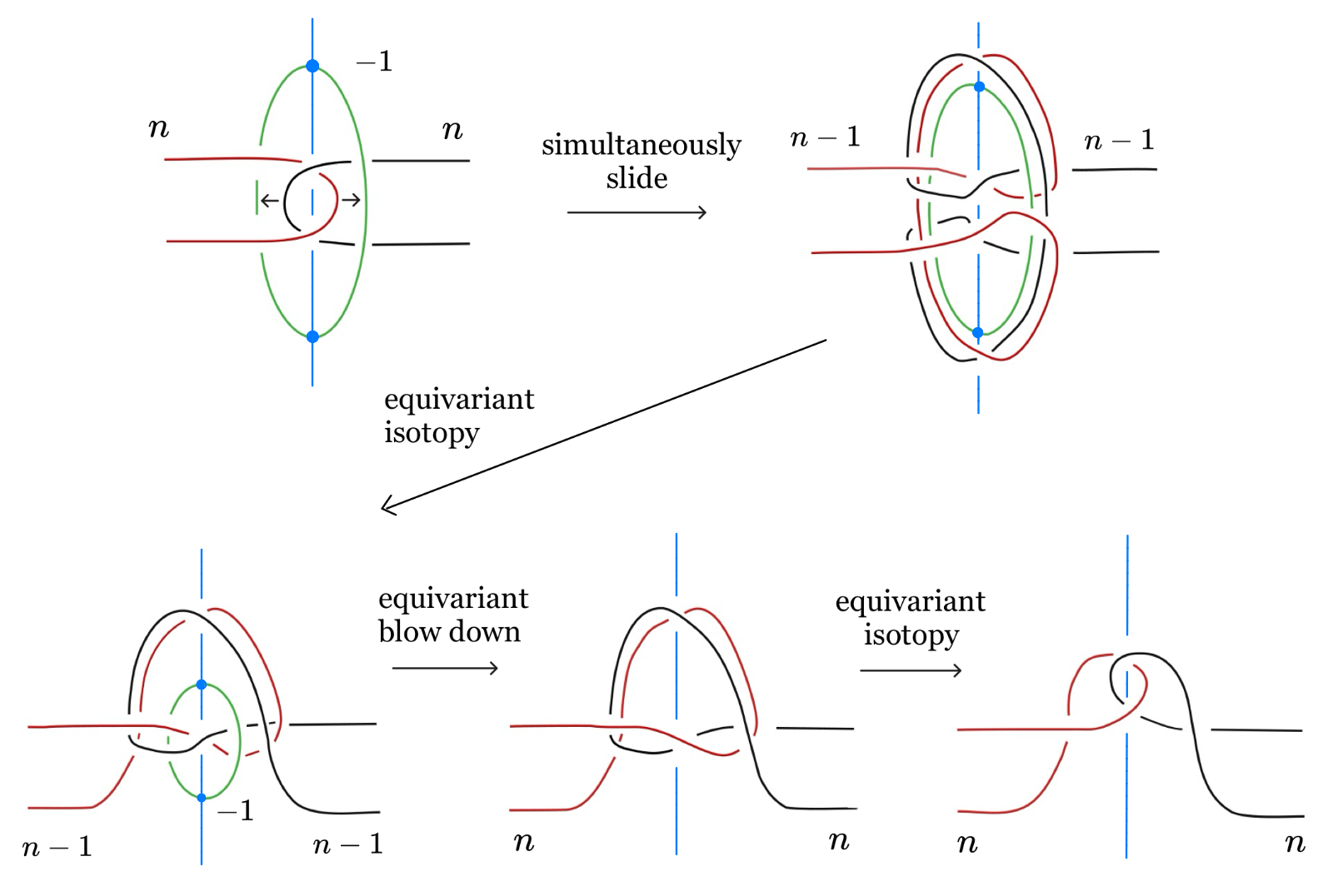}
\caption{The case where the central linking between the red and black components is positive.}
\label{operation1}
\end{center}
\end{figure}

\begin{figure}[h!]
\begin{center}
\includegraphics[width=4.7in]{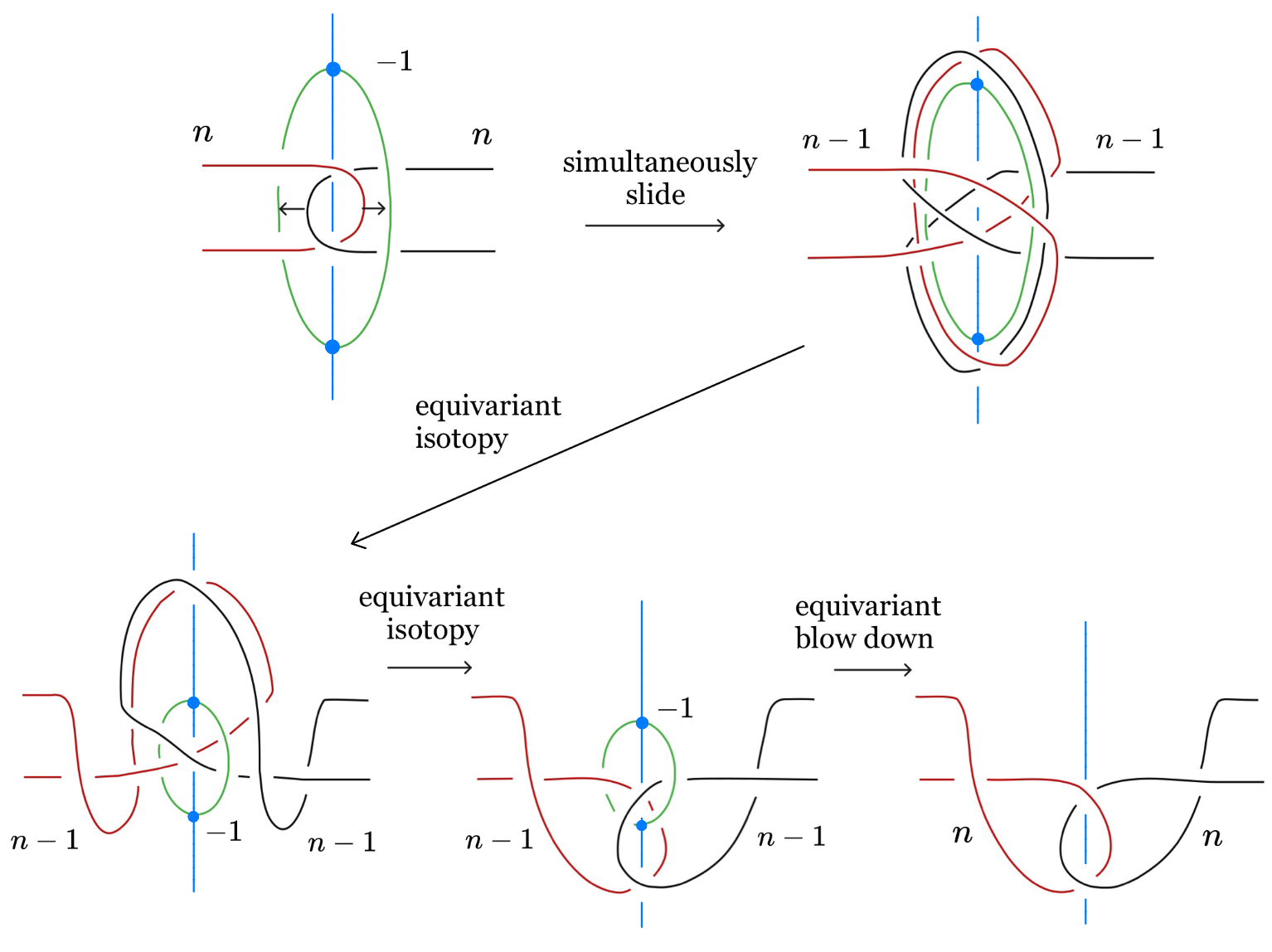}
\caption{The case where the central linking between the red and black components is negative.}
\label{operation2}
\end{center}
\end{figure}

\begin{proof}
We check the case where the central unknot has framing \(-1\).  See Figures \ref{operation1} and \ref{operation2}.
A similar deformation holds for the case where the central unknot has framing \(+1\).    
\end{proof}

\begin{proof}[Proof of Theorem \ref{thm_Y_m,n}]
First, \(Y_{m,n}\) bounds a compact contractible 4-manifold. Indeed, the Kirby diagram obtained from the surgery diagram of \(Y_{m,n}\) by replacing one of the 0-framed components with a dotted circle represents a contractible 4-manifold whose boundary is \(Y_{m,n}\).

We construct an interchanging \((-1,-1)\)-cobordism from \((Y_{m,1},\sigma)\) to \((\Sigma(2,2m+1,4m+3),\sigma)\) as shown in Figure \ref{Ycob}.

\begin{figure}[p]
\begin{center}
\includegraphics[width=5.8in]{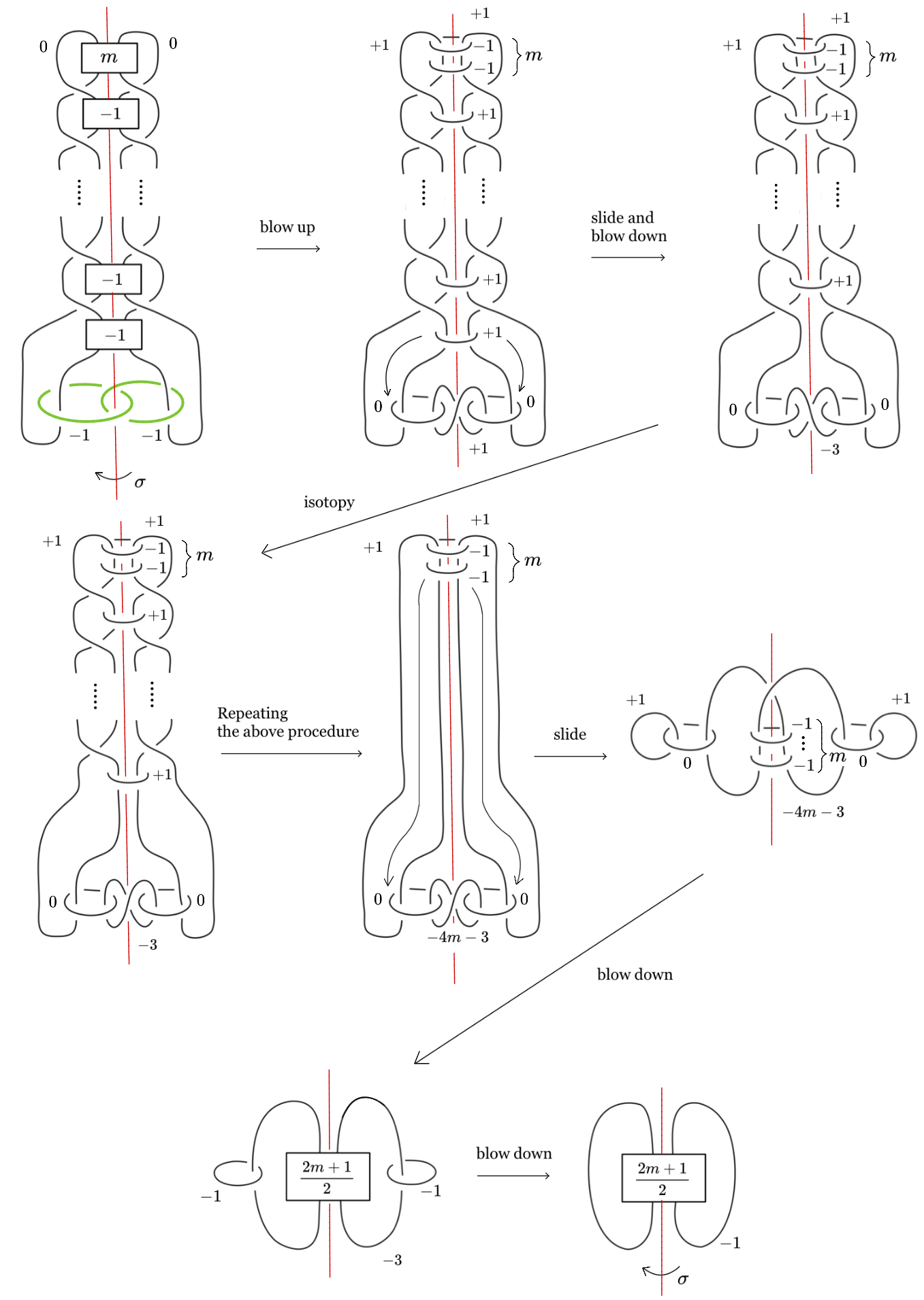}
\caption{The interchanging \((-1,-1)\)-cobordism from \((Y_{m,1},\sigma)\) to \((\Sigma (2,2m+1, 4m+3),\sigma)\) used in the proof of Theorem \ref{thm_Y_m,n}, obtained by attaching an equivariant pair of \((-1)\)-framed 2-handles along the green unknots.}
\label{Ycob}
\end{center}
\end{figure}

This is formed by attaching an equivariant pair of \((-1)\)-framed 2-handles along the green unknots as in the top left of Figure \ref{Ycob}. We see that the condition of 
Theorem \ref{thm3.2} (2) 
is satisfied by simultaneously sliding the left green unknot over the right component of \(Y_{m,1}\) and the right green unknot over the left component of \(Y_{m,1}\). 
By Theorem \ref{thm3.2} (2) and Lemma \ref{lem3.3},
    \[
    h_{\iota\circ\sigma}(Y_{m,1})\leq h_{\iota\circ\sigma}(\Sigma (2,2m+1, 4m+3))<0.
    \]
By Lemma \ref{lem3.5}, we obtain a spin\(^c\)-conjugating \((-1)\)-cobordism from \((Y_{m,n+1},\sigma)\) to \((Y_{m,n}, \sigma)\) as shown in Figure \ref{Ycob2}. By Theorem \ref{thm3.2} (1),
    \[
     h_{\iota\circ\sigma}(Y_{m,n})\leq h_{\iota\circ\sigma}(Y_{m,1})<0.
    \]
Therefore, it follows from Theorem \ref{thm3.1} that \(Y_{m,n}\) is a strong cork.

Moreover, if \(Y'_{m,n}\)
is constructed from \(Y_{m,n}\) by introducing any number of symmetric pairs of negative full twists, then \(Y'_{m,n}\) admits a sequence of interchanging \((-1,-1)\)-cobordisms to \(Y_{m,n}\). Thus \(Y'_{m,n}\) is also a strong cork.
\end{proof}

\begin{remark}
    It can be verified that \(Y_{m,1}\) is diffeomorphic to the boundary \(\partial W_m\), where \((W_m,\tau)\) is the Akbulut-Yasui cork in \cite{zbMATH05660539}. In \cite[Theorem 1.13]{Dai2020CorksIA}, it is shown that \(h_{\tau}(\partial W_m)\leq h_{\tau}(\Sigma (2,2m+1, 4m+3))<0\), which implies that \((\partial W_m,\tau)\) is a strong cork.
\end{remark}

\begin{figure}[h!]
\begin{center}
\includegraphics[width=3.8in]{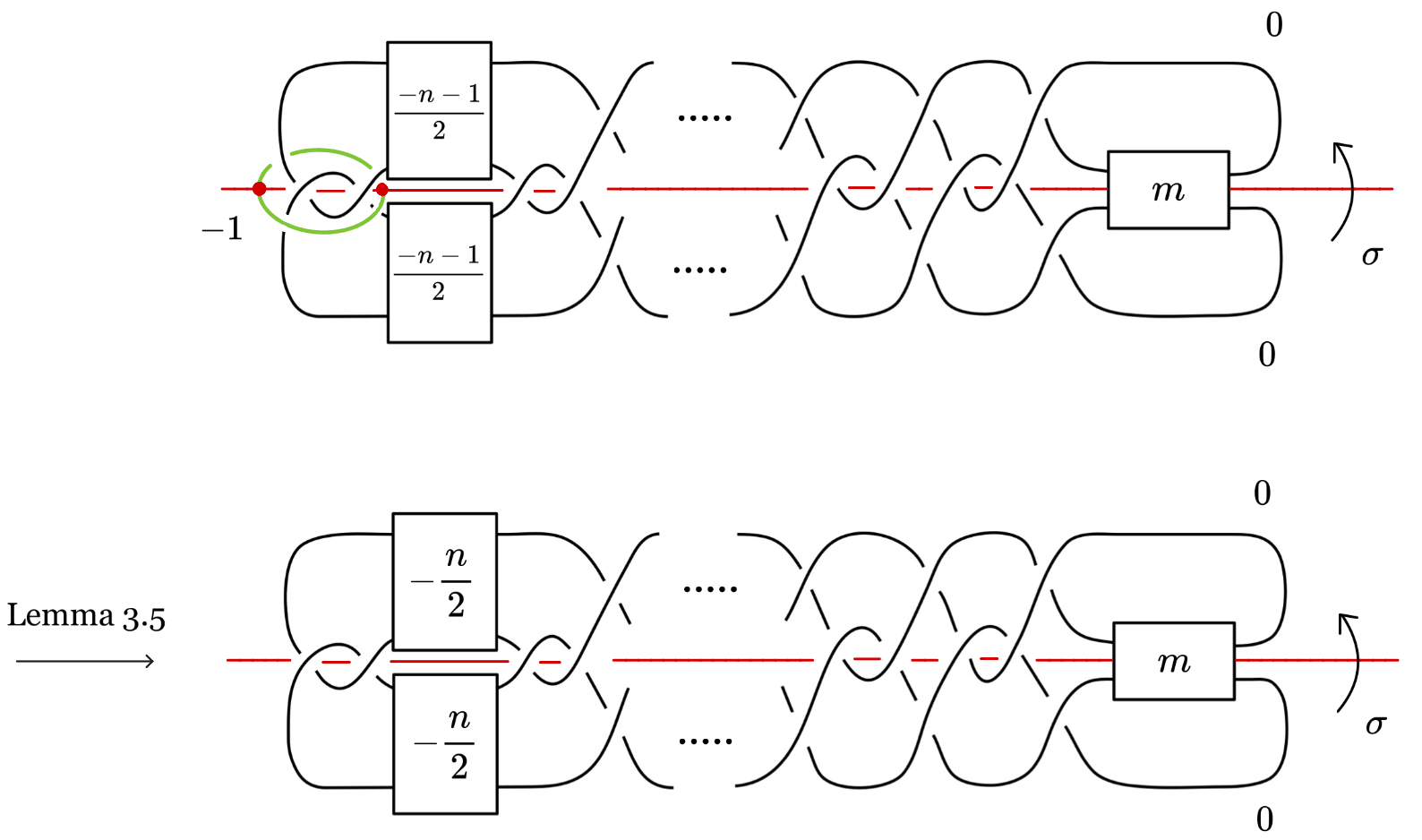}
\caption{The spin\(^c\)-conjugating \((-1)\)-cobordism from \((Y_{m,n+1},\sigma)\) to \((Y_{m,n}, \sigma)\) used in the proof of Theorem \ref{thm_Y_m,n}, obtained by attaching an equivariant \((-1)\)-framed 2-handle along the green unknot}
\label{Ycob2}
\end{center}
\end{figure}

\subsection{Instanton-theoretic methods}
Instanton-theoretic methods developed in \cite{arXiv:2309.02309} enable us to detect strong corks and treat examples inaccessible to Heegaard Floer theory.

\begin{theorem}\label{3.6}\cite[Theorem 1.1]{arXiv:2309.02309}
    Let \((Y,\tau)\) be an oriented integer homology 3-sphere equipped with an orientation-preserving involution on \(Y\). For \(s\in [-\infty, 0]\), there is a real number 
    \[
    r_s(Y,\tau)\in (0,\infty]
    \]
    which is an invariant of the diffeomorphism class of \((Y,\tau)\). Moreover, if there is an equivariant negative-definite cobordism \((W,\tilde{\tau})\) from \((Y,\tau)\) to \((Y',\tau')\) with \(H_1(W,\Z_2)=0\), then
    \[
    r_s(Y,\tau)\leq r_s(Y',\tau').
    \]
    If \(r_s(Y,\tau)<\infty\) and \(W\) is simply connected, then
    \[
    r_s(Y,\tau)< r_s(Y',\tau').
    \]
\end{theorem}

It follows from Theorem \ref{3.6} that if \(\tau\) extends over some homology ball \(W\) that \(Y\) bounds, then \(r_s(Y,\tau)=\infty\) for any \(s\in[-\infty, 0]\). Thus, this invariant \(r_s\) can be used to detect strong corks.

In contrast to Theorem \ref{thm3.2} in Heegaard Floer theory, the inequality for the \(r_s\)-invariant holds regardless of whether the cobordism is spin\(^c\)-fixing or spin\(^c\)-conjugating, as in Theorem \ref{3.6}.

The following theorem allows us to treat linear combinations using the \(r_s\)-invariant.

\begin{theorem}\cite[Theorem 7.7]{arXiv:2309.02309}\label{thm3.7}
    Let \(\{(Y_i,\tau_i)\}^{\infty}_{i=1}\) be a sequence of oriented integer homology 3-spheres equipped with orientation-preserving involutions \(\tau_i\). Assume that each \(Y_i\) bounds a compact, contractible 4-manifold, and that the fixed-point set of each \(\tau_i\) is a copy of \(S^1\), so that the equivariant connected sum operation is well-defined. Suppose that:
    \begin{enumerate}
        \item \(r_0(Y_1,\tau_1)>r_0(Y_2,\tau_2)>\cdots >r_0(Y_i,\tau_i)>\cdots\),
        \item \(r_0(Y_1,\tau_1)<\infty\),
        \item \(r_0(-Y_i,\tau_i)=\infty\) for each \(i\).
    \end{enumerate}
    Then, any nontrivial linear combination of elements in \(\{(Y_i,\tau_i)\}^{\infty}_{i=1}\) is a strong cork. 
\end{theorem}

As stated in the following lemma, the boundary of the Akbulut cork is shown to be non-trivial with respect to the \(r_0\)-invariant:

\begin{lemma}\cite[Lemma 7.1]{arXiv:2309.02309}\label{lem3.8}
    Let \(Y=S^3_{+1}(\overline{9}_{46})\) equipped with the indicated  involutions \(\tau\) and \(\sigma\) displayed in Figure \ref{9_46}. Then
    \[
    r_0(Y,\tau)<\infty,\ r_s(-Y,\tau)=\infty
    \]
    for any \(s\in[-\infty,0]\). This statement holds with \(\tau\) replaced by \(\sigma\).
\end{lemma}

\subsection{Proof of Theorems \ref{thm_Z_m,n}, \ref{thm_K_n}, and \ref{thm_Y_1,n}}
\begin{proof}[Proof of Theorem \ref{thm_Z_m,n}]
First, \(Z_{m,n}\) bounds a compact contractible 4-manifold. Indeed, the Kirby diagram obtained from the surgery diagram of \(Z_{m,n}\) by replacing one of the 0-framed components with a dotted circle represents a contractible 4-manifold whose boundary is \(Z_{m,n}\).

Using the operation in Lemma \ref{lem3.5}, we obtain an equivariant negative-definite cobordism \(W_{m,n}\) from \((Z_{m,n+1},\tau)\) to \((Z_{m,n},\tau)\), as shown in Figure \ref{Z_cob}. Similarly, we obtain an equivariant negative-definite cobordism from \((Z_{m+1,n},\tau)\) to \((Z_{m,n},\tau)\).
Since \((Z_{1,1},\tau)\) is equivariantly diffeomorphic to \((S^3_{+1}(\overline{9}_{46}),\tau)\), by Theorem \ref{3.6} and Lemma \ref{lem3.8},
    \[
    r_0(Z_{m,n}, \tau)\leq r_0(Z_{1,1},\tau)<\infty.
    \]
Therefore, for any \(m,n\geq1\), \((Z_{m,n},\tau)\) is a strong cork.

\begin{figure}[h]
\begin{center}
\includegraphics[width=6.3in]{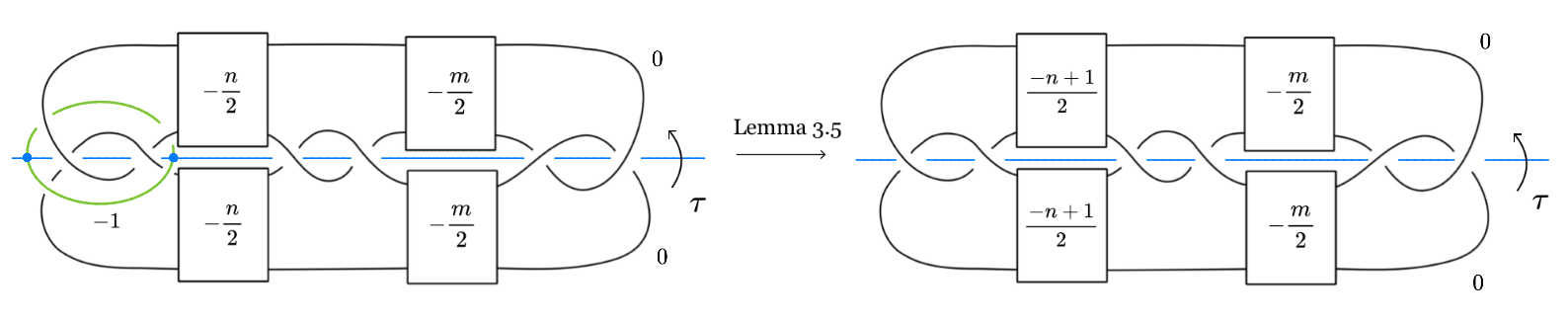}
\caption{The equivariant negative-definite cobordism \(W_{m,n}\) from \((Z_{m,n+1},\tau)\) to \((Z_{m,n},\tau)\) used in the proof of Theorem \ref{thm_Z_m,n}, obtained by attaching an equivariant \((-1)\)-framed 2-handle along the green unknot.}
\label{Z_cob}
\end{center}
\end{figure}

\begin{claim}\label{claim}
    This cobordism \(W_{m,n}\) from \((Z_{m,n+1},\tau)\) to  \((Z_{m,n},\tau)\) is simply connected.
\end{claim}

\begin{figure}[h!]
\begin{center}
\includegraphics[width=3.0in]{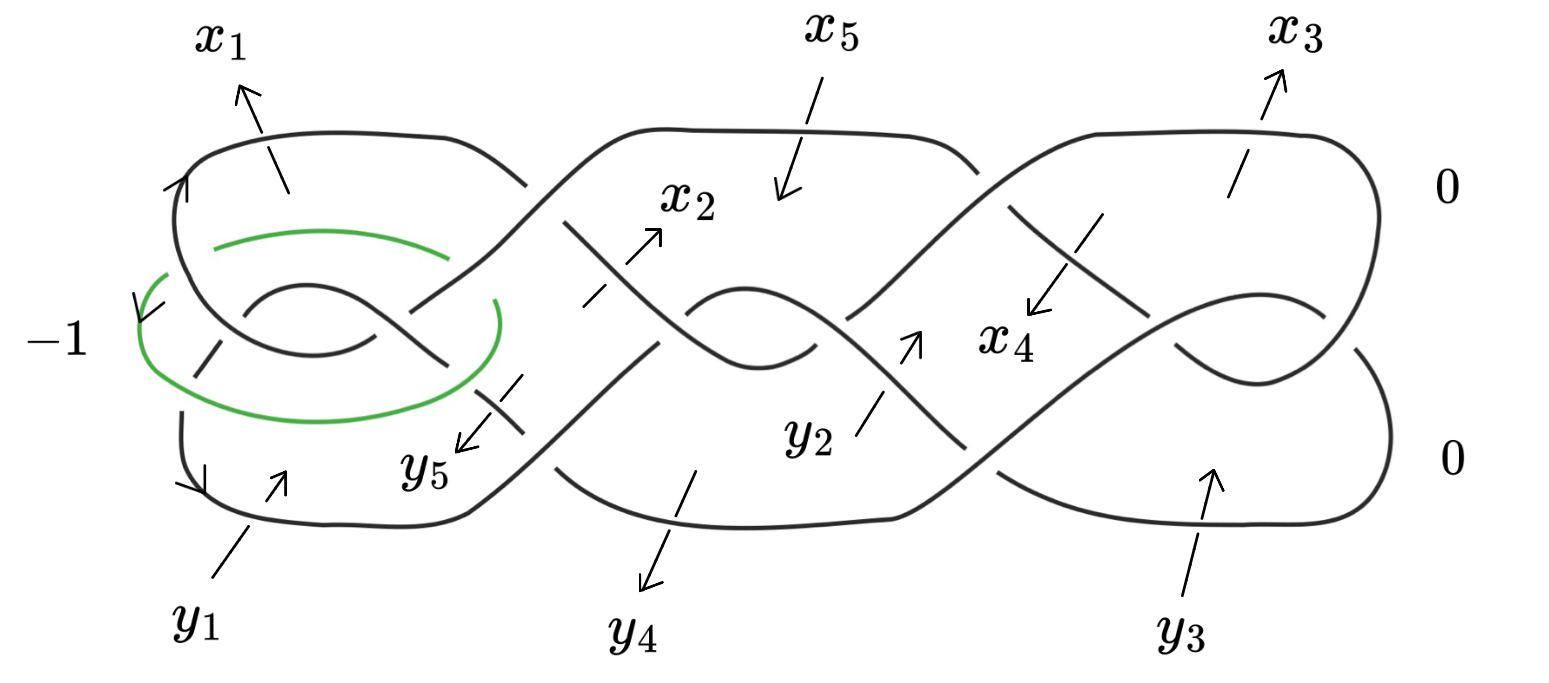}
\caption{The cobordism \(W_{1,1}\) and the Wirtinger presentation of the surgery link for \(Z_{1,2}.\)}
\label{simpconn}
\end{center}
\end{figure}

\begin{proof}
    First, we verify the case \(m, n=1\). The generators of \(\pi_1(W_{1,1})\) are \(x_1,\cdots ,x_5, y_1,\cdots ,y_5\) as indicated in Figure \ref{simpconn}. The relators are given by:
    \begin{align}
    y^{-1}_1 x^{-1}_1 y_5 x_1=1\\
    x^{-1}_1 y^{-1}_5 x_5 y_5=1\\
    x^{-1}_2 x^{-1}_5 x_1 x_5=1\\
    y^{-1}_5 y^{-1}_1 y_4 y_1=1\\
    y^{-1}_2 x^{-1}_1 y_1 x_2=1\\
    x^{-1}_3 y^{-1}_2 x_2 y_2=1\\
    x^{-1}_5 x^{-1}_3 x_4 x_3=1\\
    y^{-1}_3 y^{-1}_4 y_2 y_4=1\\
    y^{-1}_4 x^{-1}_4 y_4 x_3=1\\
    x^{-1}_3 y^{-1}_4 x_3 y_3=1\\
    y_5 x_5 y_2 x^{-2}_3 y^{-1}_4 x_3=1\\
    x_1 x_2 y_4 y^{-2}_3 x^{-1}_3 y_1=1\\
    x_1 x^{-1}_5=1
    \end{align}
    The equations (1)-(10) correspond to the relators for the fundamental group of the complement of the surgery link for \(Z_{1,2}\). The equations (11) and (12) correspond to the 0-framed longitudes of the components of the diagram of \(Z_{1,2}\), respectively. The equation (13) corresponds to the green attaching sphere of the 2-handle of \(W_{1,1}\).
    
    From (13), we have \(x_1=x_5\), and let us denote this element by \(x\). If we set \(y=y_5\), then (2) implies \(xy=yx\). Thus, from (1), we have \(y_1=y\). By (3), \(x_2=x\), and by (4), \(y_4=y\). By (5), \(y_2=y\), so by (6), \(x_3=x\). By (7), \(x_4=x\), and by (8), \(y_3=y\). In summary, from (1)-(8) and (13), we obtain
    \begin{align}    
    x=x_1=x_2=x_3=x_4=x_5,\ y=y_1=y_2=y_3=y_4=y_5,\ xy=yx .
    \end{align}
    Substituting these into (11) and (12), we obtain
    \[
    x=y=1.
    \]
    Therefore, \(W_{1,1}\) is simply connected.
    
    Note that since \(Z_{1,2}\) is an integer homology 3-sphere, its first homology group is trivial. Since the first homology group is the abelianization of the fundamental group, the relations in (14) imply \(x=y=1\) immediately, without the need to substitute into (11) and (12).

    Even when additional half twists are added, we can compute inductively in a similar way. Consider the case where a half twist is added as shown in Figure \ref{halftwists}. For the left case, if we assume \(x=x_{k+1}=x_l\), then the relator \(x^{-1}_{k+1} x^{-1}_{l} x_k x_l=1\) implies \(x=x_k\). Similarly, for the right case, assuming \(x=x_{k}=x_l\), the relator \(x^{-1}_{k+1} x^{-1}_l x_k x_l=1\) implies \(x=x_{k+1}\). Thus, it is shown inductively that \(W_{m,n}\) is simply connected for any \(m\) and \(n\).
\end{proof}

\begin{figure}[h!]
\begin{center}
\includegraphics[width=2.7in]{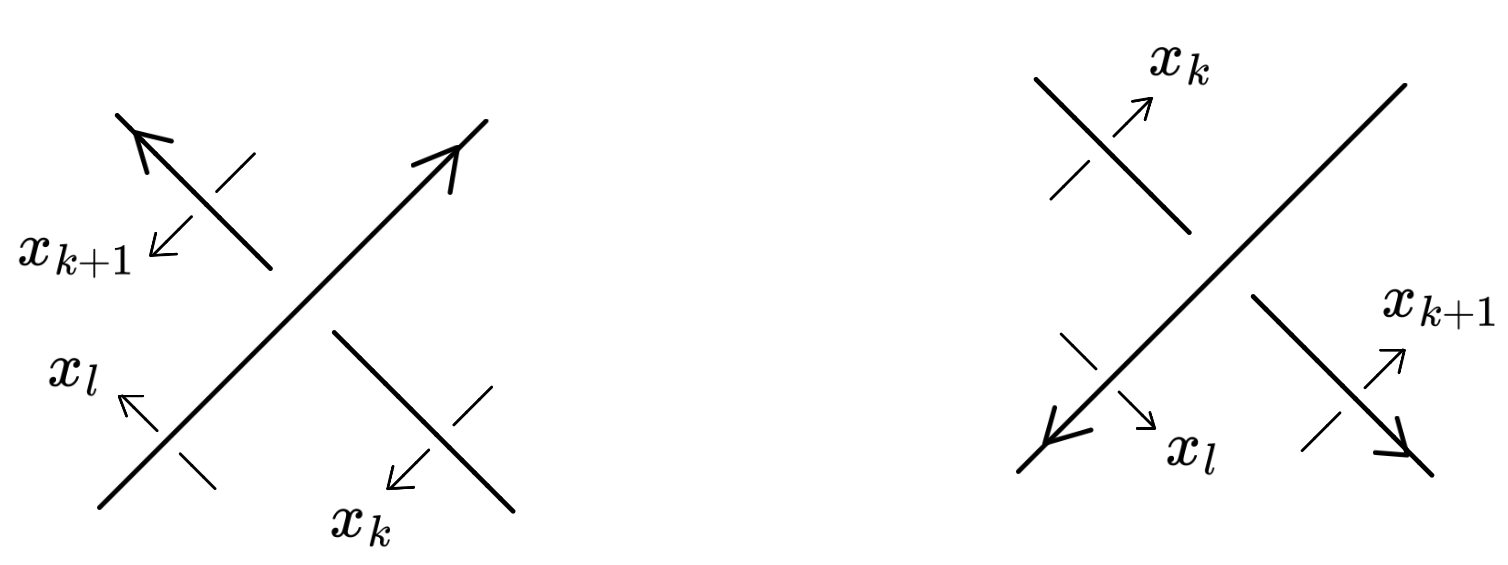}
\caption{}
\label{halftwists}
\end{center}
\end{figure}
    
    By Claim \ref{claim}, we have \(r_0(Z_{m,{n+1}},\tau)<r_0(Z_{m,n},\tau)\). Considering \(W_{m,n}\) to be the cobordism from \((-Z_{m,{n}},\tau)\) to  \((-Z_{m,{n+1}},\tau)\), we obtain
    \[
    \infty =r_0(-Z_{1,1},\tau)\leq r_0(-Z_{m,1},\tau)\leq r_0(-Z_{m,2},\tau)\leq\cdots\leq r_0(-Z_{m,n},\tau)\leq\cdots.
    \]
    Thus, \(r_0(-Z_{m,n},\tau)=\infty\). Therefore, any nontrivial linear combination of elements in \(\{(Z_{m,n},\tau)\}_{n\in\N}\) is a strong cork by Theorem \ref{thm3.7}. 

    Similarly, the cobordism from \((Z_{m+1,n},\tau)\) to \((Z_{m,n},\tau)\) is also simply connected, so any nontrivial linear combination of elements in \(\{(Z_{m,n},\tau)\}_{m\in\N}\) is a strong cork.
\end{proof}

\begin{remark}
   When \(m\) and \(n\) are odd, one can also prove that
   \[
   h_{\tau}(Z_{m,n})\leq h_{\tau}(Z_{1,1})<0
   \]
   and \((Z_{m,n},\tau)\) is strong by considering the  interchanging \((-1,-1)\)-cobordism shown in Figure \ref{C_cob2}.
   However, since the cobordism used in the proof of Theorem \ref{thm_Z_m,n} is spin\(^c\)-conjugating, we cannot obtain the inequality for \(h_{\tau}\). Thus, Theorem \ref{thm_Z_m,n} cannot be immediately shown using Heegaard Floer theory when \(m\) and \(n\) are even.
\end{remark}

\begin{figure}[h!]
\begin{center}
\includegraphics[width=6.0in]{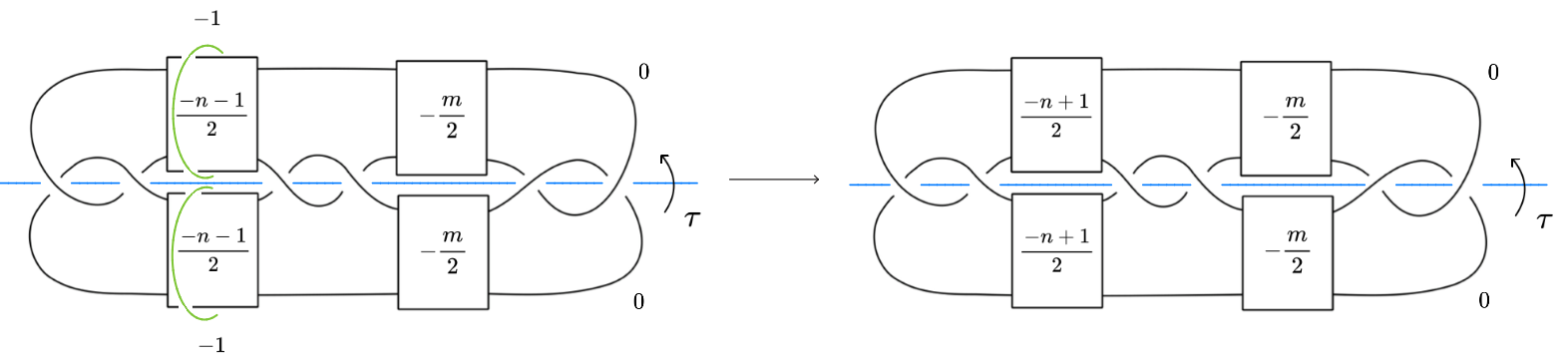}
\caption{The interchanging \((-1,-1)\)-cobordism from \((Z_{m,n+2},\tau)\) to \((Z_{m,n},\tau)\) obtained by attaching an equivariant pair of \((-1)\)-framed 2-handles along the green unknots.}
\label{C_cob2}
\end{center}
\end{figure}

\begin{remark}\label{rem}
    \((Z_{m,n}\# (-Z_{m,n}),\tau)\) is not a strong cork, but \(\tau\) does not extend over any contractible 4-manifold bounded by \(Z_{m,n}\# (-Z_{m,n})\). If \(\tau\) were to extend over some contractible 4-manifold \(W\) that \(Z_{m,n}\# (-Z_{m,n})\) bounds, then we would obtain a simply-connected, equivariant definite cobordism \(W'\) from \((Z_{m,n},\tau)\) to itself. As in the proof of Theorem \ref{thm_Z_m,n}, \(r_0(Z_{m,n},\tau)<\infty\). By Theorem \ref{3.6}, \(r_0(Z_{m,n},\tau)<r_0(Z_{m,n},\tau)\), which is a contradiction.
\end{remark}

\begin{proof}[Proof of Theorem \ref{thm_K_n}]
    It suffices to prove that \(\{(S^3_{1/m}(K_n),\tau)\}_{m\in\N}\) satisfies the assumptions of Theorem \ref{thm3.7}. 
    \begin{enumerate}
        \item The cobordism from \((S^3_{1/{(m+1)}}(K_n),\tau)\) to \((S^3_{1/m}(K_n),\tau)\) displayed in Figure \ref{K_n_cob} is simply connected, equivariant negative-definite, which can be proved similarly to \cite[Theorem 5.12]{Nozaki2019FilteredIF}.
        \item Since \((S^3_{+1}(K_n),\tau)\) is equivariantly diffeomorphic to \((Z_{1,n},\tau)\), by Theorem \ref{thm_Z_m,n},
        \[
        r_0(S^3_{+1}(K_n),\tau)=r_0(Z_{1,n}, \tau)<\infty.
        \]
        \item By \cite[Theorem 6.3]{arXiv:2309.02309}, \(r_0(-S^3_{1/m}(K_n),\tau)=\infty\) for any \(m\geq 1\).
    \end{enumerate}
    The same holds for the involution \(\sigma\).
\end{proof}

\begin{figure}[h!]
\begin{center}
\includegraphics[width=5.0in]{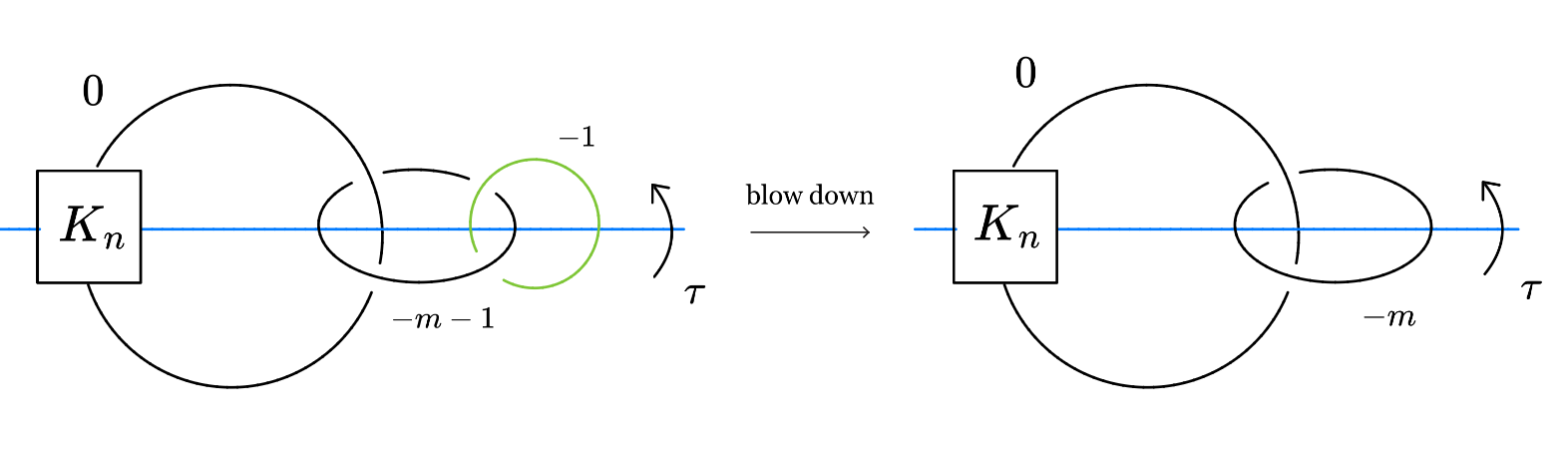}
\caption{The simply-connected, equivariant negative-definite cobordism from \((S^3_{1/{(m+1)}}(K_n),\tau)\) to \((S^3_{1/m}(K_n),\tau)\) used in the proof of Theorem \ref{thm_K_n}, obtained by attaching an equivariant \((-1)\)-framed 2-handle along the green unknot.}
\label{K_n_cob}
\end{center}
\end{figure}

\begin{remark}
    Theorem \ref{thm_Y_m,n} cannot immediately be proved by instanton theory using the cobordism in the proof of Theorem \ref{thm_Y_m,n} because the \(r_s\)-invariants are trivial for Brieskorn spheres, as discussed in \cite[Section 7.2]{arXiv:2309.02309}.
\end{remark}

\begin{proof}[Proof of Theorem \ref{thm_Y_1,n}]
    Since \((Y_{1,1},\sigma)\) is equivariantly diffeomorphic to \((S^3_{+1}(\overline{9}_{46}),\sigma)\), Lemma \ref{lem3.8} implies that
    \[
    r_0(Y_{1,1},\sigma)=r_0(S^3_{+1}(\overline{9}_{46}),\sigma)<\infty.
    \]
    Similarly to the proof of Theorem \ref{thm_Z_m,n}, there is a simply-connected, equivariant negative-definite cobordism from \((Y_{1,n+1},\sigma)\) to \((Y_{1,n},\sigma)\). The result follows from Theorem \ref{thm3.7}.
\end{proof}

\section{Appendix}\label{appendix}
In this appendix, we prove Proposition \ref{prop_sakuma}, which states that the knot \(K_n\) has two distinct strong involutions \(\tau\) and \(\sigma\) displayed in Figure \ref{K_n}.
\begin{lemma}\label{A1}
    For each \(n\in\N\), the knot displayed on the left side of Figure \ref{K_n} is isotopic to the one displayed on the right. 
\end{lemma}

\begin{proof}
    See Figure \ref{isotopy}.
\end{proof}

It remains to prove that the involutions \(\tau\) and \(\sigma\) are distinct. We prove that \((K_n,\tau)\) and \((K_n,\sigma)\) are not Sakuma equivalent by computing their Sakuma \(\eta\)-polynomials introduced in \cite{1984OnSI}.

We review the definition of the \(\eta\)-polynomial of a strongly invertible knot \((K,\tau)\). 
Let \(l\) be a preferred longitude of \(K\) such that \(\tau (l)\cap l=\emptyset\). Set \(O=p(\textrm{Fix}(\tau))\) and \(L=p(l)\), where \(p:S^3\rightarrow S^3/\tau\cong S^3\) is a projection. Then \(L(K,\tau)=O\cup L\) is a two-component link in \(S^3\) and \(\textrm{lk}(O,L)=0\). The \(\eta\)-polynomial of \((K,\tau)\) is defined as the Laurent polynomial
\[
  \eta_{(K,\tau)}(t)=\sum_{i=-\infty}^{\infty} \textrm{lk}(\tilde{L'},t^i \tilde{L})t^i\in\Z[t^{\pm 1}],
\]
where \(\tilde{L}\) is the lift of \(L\) to the infinite cyclic cover \(\widetilde{E(O)}\) of \(E(O)=S^3\setminus O\), \(\tilde{L'}\) is the lift of a preferred longitude \(L'\) of \(L\) near \(\tilde{L}\), and \(t\) is a generator of the covering transformation group of \(\widetilde{E(O)}\). The \(\eta\)-polynomial \(\eta_{(K,\tau)}(t)\) is an invariant of the Sakuma equivalence class of \((K,\tau)\).

\begin{lemma}
    For each \(n\in\N\), \(\eta_{(K_n,\tau)}(t)\neq\eta_{(K_n, \sigma)}(t)\).
\end{lemma}

\begin{figure}[h!]
\begin{center}
\includegraphics[width=3.8in]{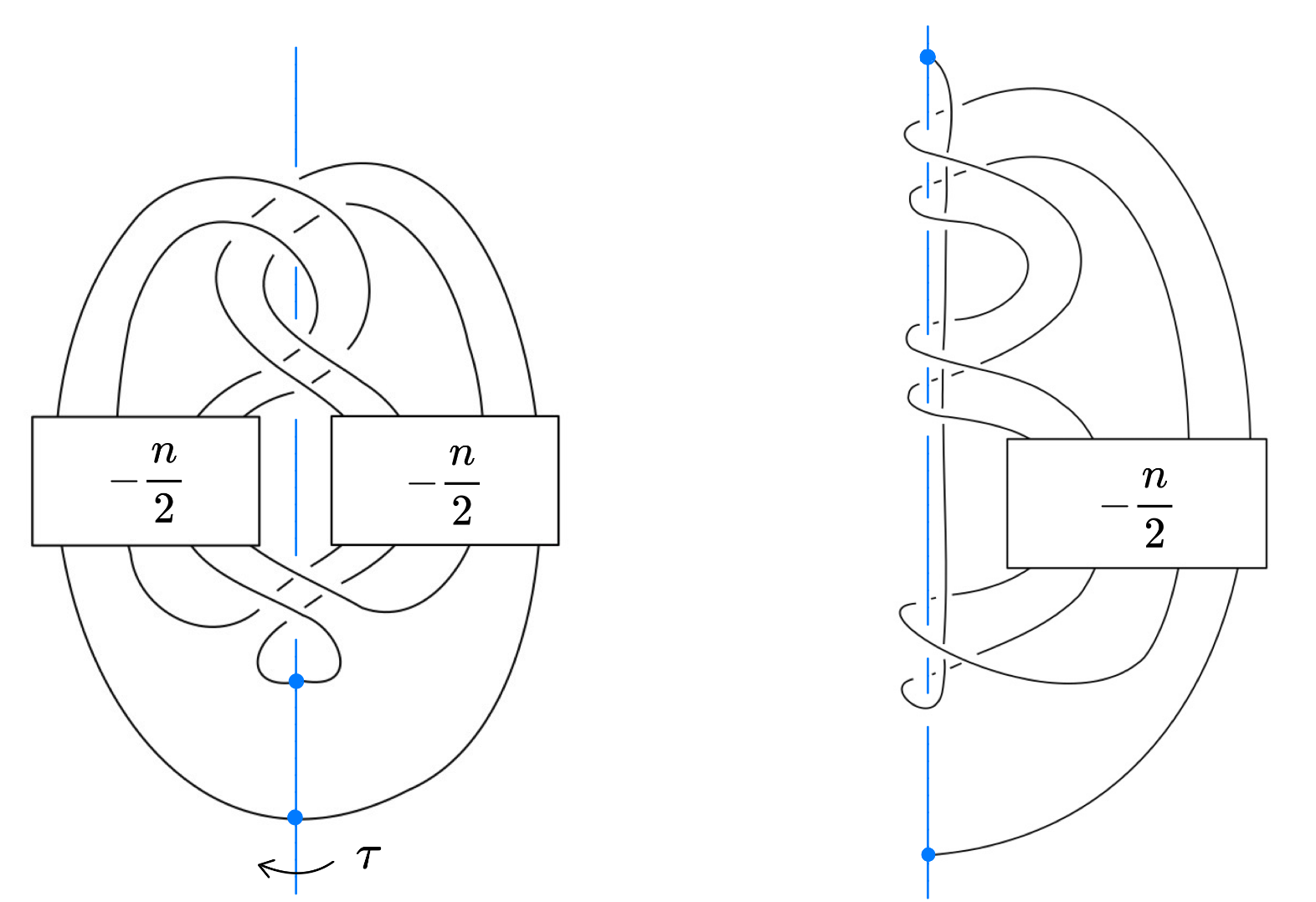}
\caption{Left: \((K_n,\tau)\). Right: The intermediate step to obtain the pseudo-fundamental region.}
\label{tau_theta}
\end{center}
\end{figure}

\begin{figure}[h!]
\begin{center}
\includegraphics[width=4.7in]{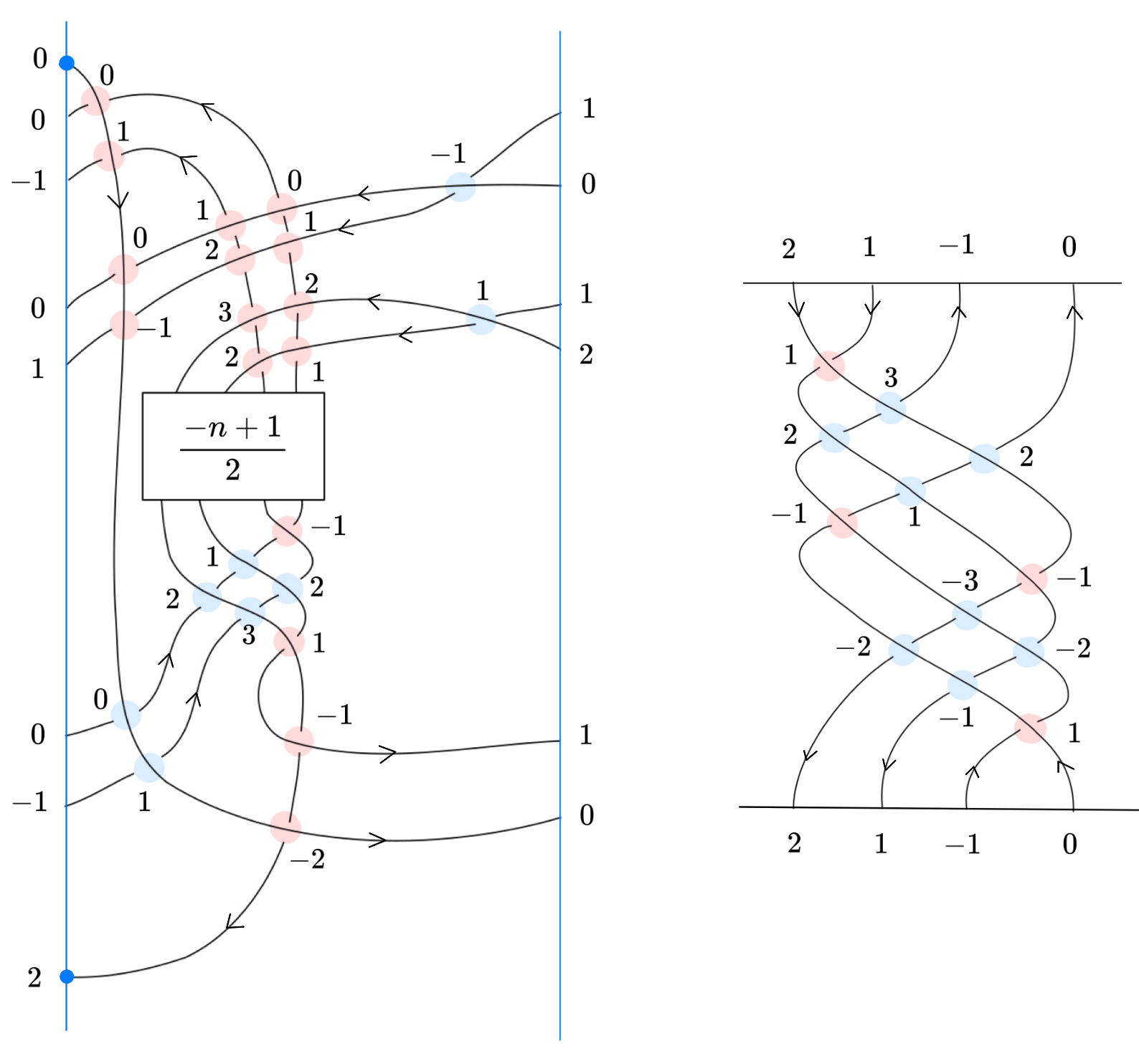}
\caption{Left: The pseudo-fundamental region of \((K_n,\tau)\) where \(n\) is odd. Indices are assigned to the arcs according to Sakuma's algorithm. Starting with 0, the index is increased by 1 if the arc starts from the right, and decreased by 1 if it starts from the left.
At each crossing, red and blue circles indicate positive and negative signs, respectively, and the number represents the index of the over-arcs minus the index of the under-arcs. Right: A detail of one of the full twists in the box on the left.}
\label{tau_odd}
\end{center}
\end{figure}

\begin{proof}
    We use the algorithm described in \cite[Section 2]{1984OnSI}. First, we compute the Sakuma \(\eta\)-polynomial \(\eta_{(K_n,\tau)}(t)\) of \((K_n,\tau)\). 
    
    \underline{The case where \(n\) is odd}: Via the process in Figure \ref{tau_theta}, we obtain the pseudo-fundamental region of \((K_n,\tau)\) in Figure \ref{tau_odd}, which gives the calculation of the \(\eta\)-polynomial.    
    \[
    \tilde{\eta}_{(K_n,\tau)}
    =-\dfrac{n-1}{2}(x_{-3}+x_3)+(-n+2)(x_{-2}+x_2)
    +\dfrac{n+3}{2}(x_{-1}+x_1)+2x_0.
    \]
    Putting \(x_i=t^{i-1}-2t^i+t^{i+1}\),
    \[
    \eta'_{(K_n,\tau)}
    =(n+1)+(-2n-1)(t^{-1}+t)+(2n-2)(t^{-2}+t^2)+(t^{-3}+t^3)-\dfrac{n-1}{2}(t^{-4}+t^4).
    \]
    If \(\eta'=[a_0,a_1,a_2,a_3,\cdots\), then,
    \[
    \eta=[-2\Sigma_{j\geq 1}a_{2j},\ -\Sigma_{j\geq 1}a_{2j+1},\ a_2,a_3,\cdots.
    \]
    Here, \([a_0,a_1,a_2,\cdots, a_n\) represents the polynomial \(a_0+a_1(t^{-1}+t)+a_2(t^{-2}+t^2)+\cdots +a_n(t^{-n}+t^n)\). Hence, the \(\eta\)-polynomial of \((K_n,\tau)\) is
    \[
    \eta_{(K_n,\tau)}=[-3n+3,\ -1,\ 2n-2,\ 1,\ -\dfrac{n-1}{2}.
    \]
    
     \underline{The case where \(n\) is even}: Figure \ref{tau_even} shows the pseudo-fundamental region of \((K_n,\tau)\) where \(n\) is even.
    \[
    \tilde{\eta}_{(K_n,\tau)}=-\dfrac{n}{2}(x_{-5}+x_5)+(n-1)(x_{-3}+x_3)+(-n+1)(x_{-2}+x_2)+\left(-\dfrac{n}{2}+1\right)(x_{-1}+x_1).
    \]
 Putting \(x_i=t^{i-1}-2t^i+t^{i+1}\),
    \[
    \eta'_{(K_n,\tau)}
    =(-n+2)-(t^{-1}+t^1)+\left(\dfrac{5}{2}n-2\right)(t^{-2}+t^2)
    +(-3n+3)(t^{-3}+t^3)+\left(\dfrac{n}{2}-1\right)(t^{-4}+t^4)+n(t^{-5}+t^5)-\dfrac{n}{2}(t^{-6}+t^6).
    \]
    Hence,
    \[
    \eta_{(K_n,\tau)}=[-5n+6,\ 2n-3,\ \dfrac{5}{2}n-2,\ -3n+3,\ \dfrac{n}{2}-1,\ n,\ -\dfrac{n}{2}.
    \]
    
\begin{figure}[h!]
\begin{center}
\includegraphics[width=4.8in]{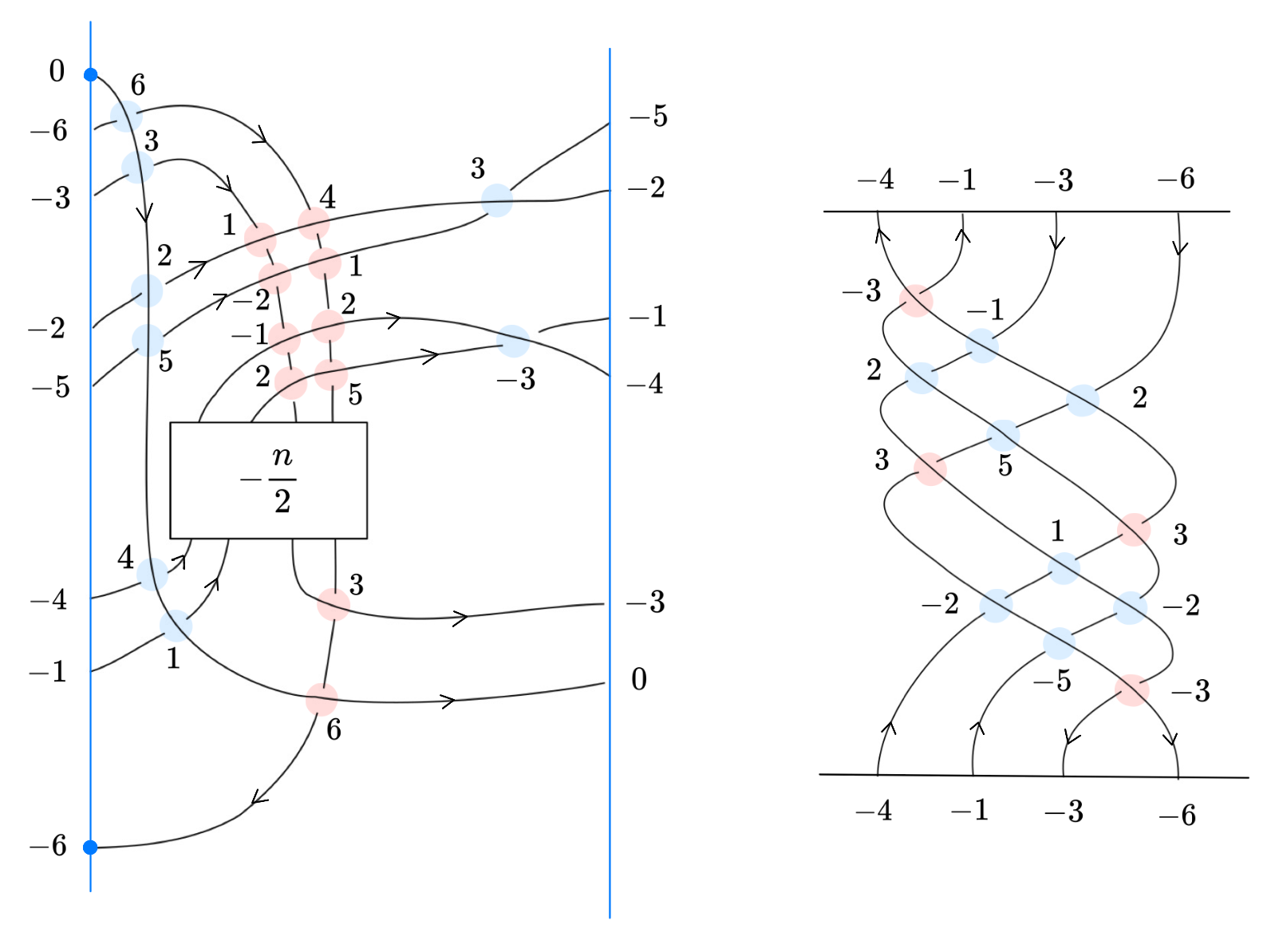}
\caption{Left: The pseudo-fundamental region of \((K_n,\tau)\) where \(n\) is even. Right: A detail of one of the full twists in the box on the left.}
\label{tau_even}
\end{center}
\end{figure}

    Next, we compute the \(\eta\)-polynomial \(\eta_{(K_n,\sigma)}(t)\) of \((K_n,\sigma)\). 
    
    \underline{The case where \(n\) is odd}: Via the process in Figure \ref{sigma_theta}, we obtain the pseudo-fundamental region of \((K_n,\sigma)\) in Figure \ref{sigma_odd}.    
    \[
    \tilde{\eta}_{(K_n,\sigma)}
    =\dfrac{n-1}{2}(x_{-5}+x_5)+(-n+2)(x_{-3}+x_3)+(n-2)(x_{-2}+x_2)+\dfrac{n-3}{2}(x_{-1}+x_1).
    \]
    Putting \(x_i=t^{i-1}-2t^i+t^{i+1}\),
    \[
    \eta'_{(K_n,\sigma)}
    =(n-3)+(t^{-1}+t)+\dfrac{-5n+9}{2}(t^{-2}+t^2)+(3n-6)(t^{-3}+t^3)+\dfrac{-n+3}{2}(t^{-4}+t^4)+(-n+1)(t^{-5}+t^5)+\dfrac{n-1}{2}(t^{-6}+t^6).
    \]
    Hence,
    \[
    \eta_{(K_n,\sigma)}
    =[5n-11,\ -2n+5,\ \dfrac{-5n+9}{2},\ 3n-6,\ \dfrac{-n+3}{2},\ -n+1,\ \dfrac{n-1}{2}.
    \]

\begin{figure}[h!]
\begin{center}
\includegraphics[width=4.0in]{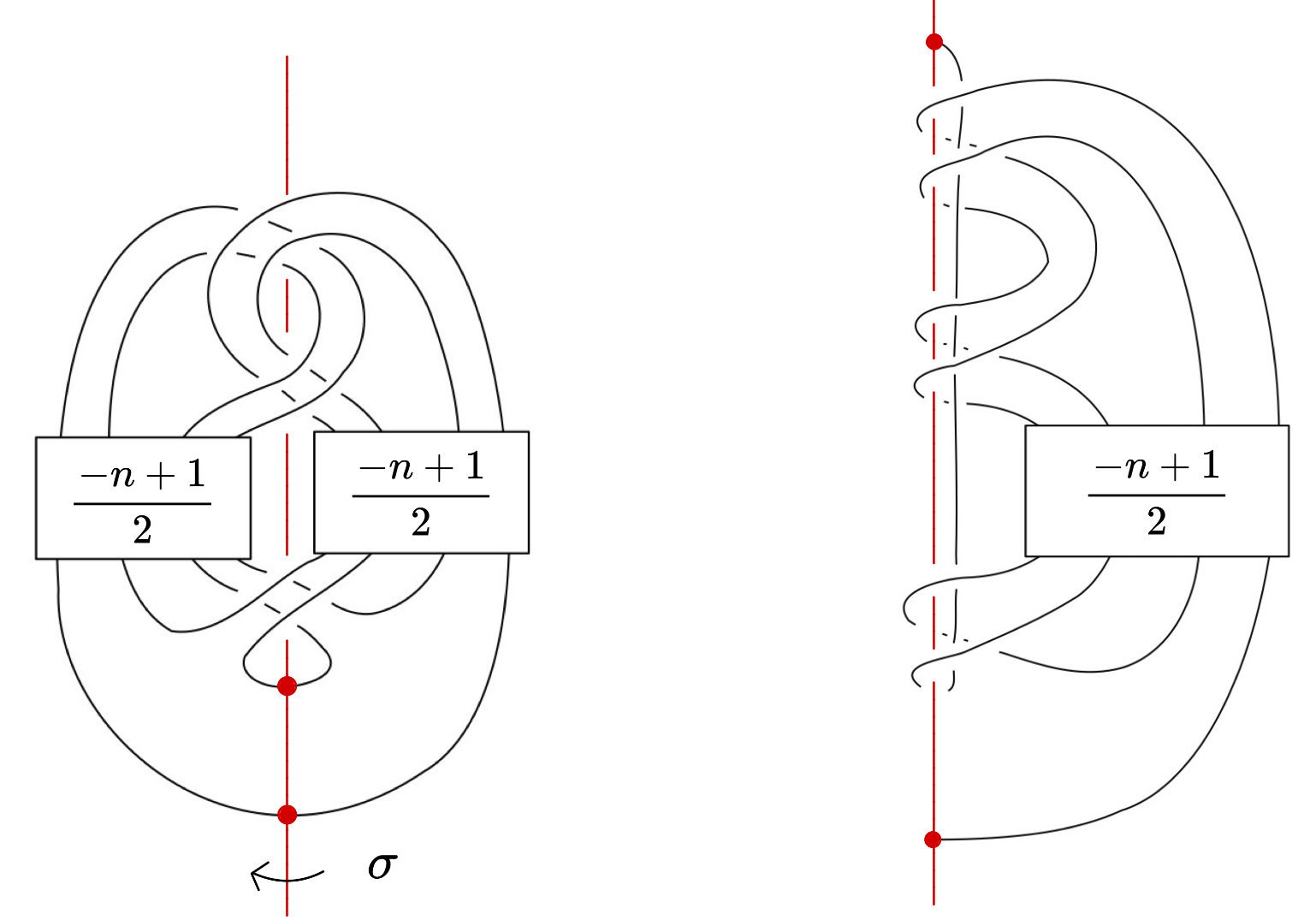}
\caption{Left: \((K_n,\sigma)\). Right: The intermediate step to obtain the pseudo-fundamental region.}
\label{sigma_theta}
\end{center}
\end{figure}

\begin{figure}[h!]
\begin{center}
\includegraphics[width=4.8in]{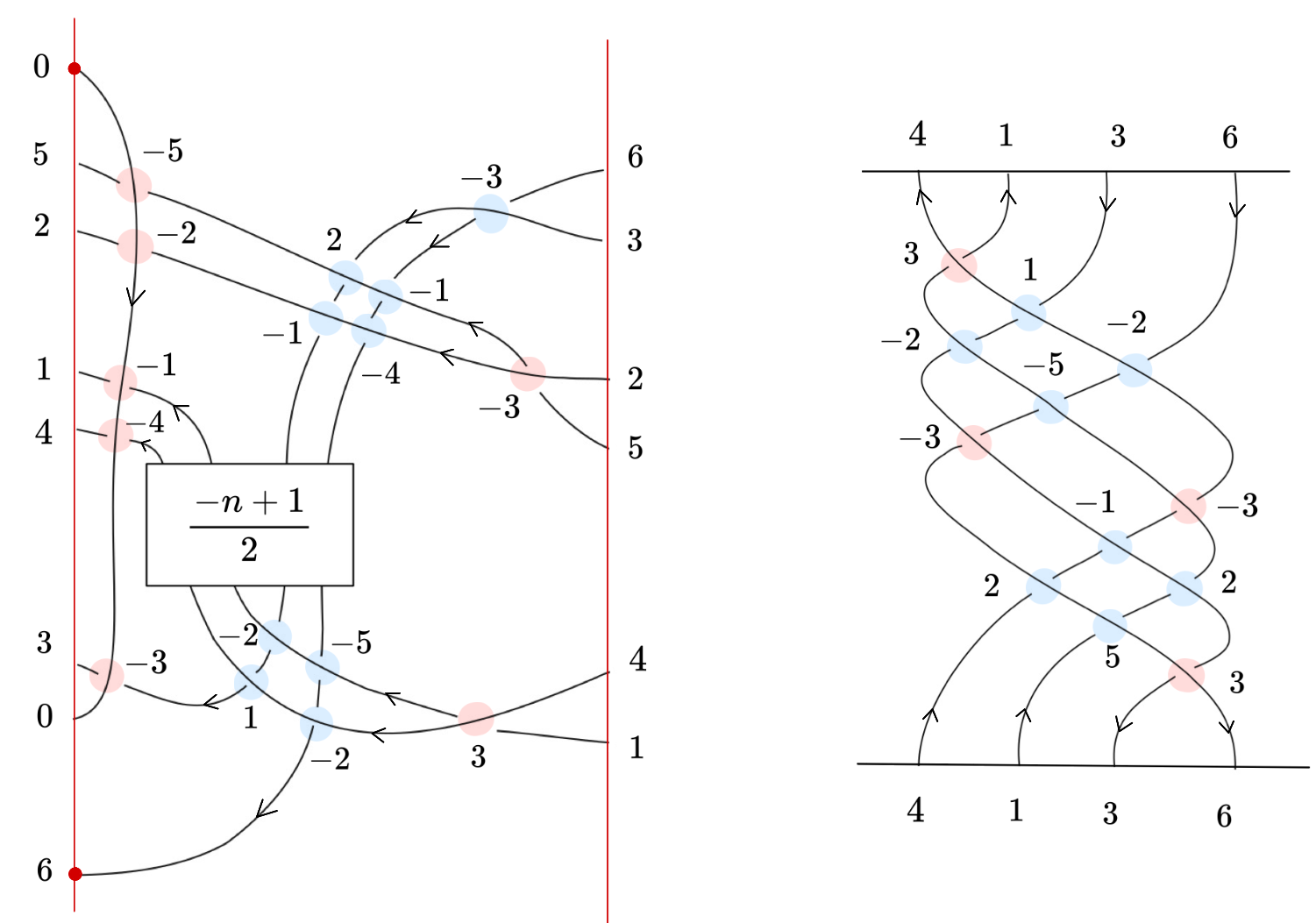}
\caption{Left: The pseudo-fundamental region of \((K_n,\sigma)\) where \(n\) is odd. Right: A detail of one of the full twists in the box on the left.}
\label{sigma_odd}
\end{center}
\end{figure}
 
    \underline{The case where \(n\) is even}: Figure \ref{sigma_even} shows the pseudo-fundamental region of \((K_n,\sigma)\) where \(n\) is even.
    
    \[
     \tilde{\eta}_{(K_n,\sigma)}
    =\dfrac{n-4}{2}(x_{-3}+x_3)+(n-5)(x_{-2}+x_2)-\dfrac{n}{2}(x_{-1}+x_1)-2x_0.
    \]
    Putting \(x_i=t^{i-1}-2t^i+t^{i+1}\),
    \[
    \eta'_{(K_n,\sigma)}
    =\dfrac{n-4}{2}(t^{-4}+t^4)-(t^{-3}+t^3)+(-2n+8)(t^{-2}+t^2)+(2n-7)(t^{-1}+t)+(-n+4).
    \]
    Hence,
    \[
        \eta_{(K_n,\sigma)}
        =[3n-12,\ 1,\ -2n+8,\ -1,\ \dfrac{n-4}{2}.
    \]
\end{proof}

\begin{figure}[t]
\begin{center}
\includegraphics[width=4.8in]{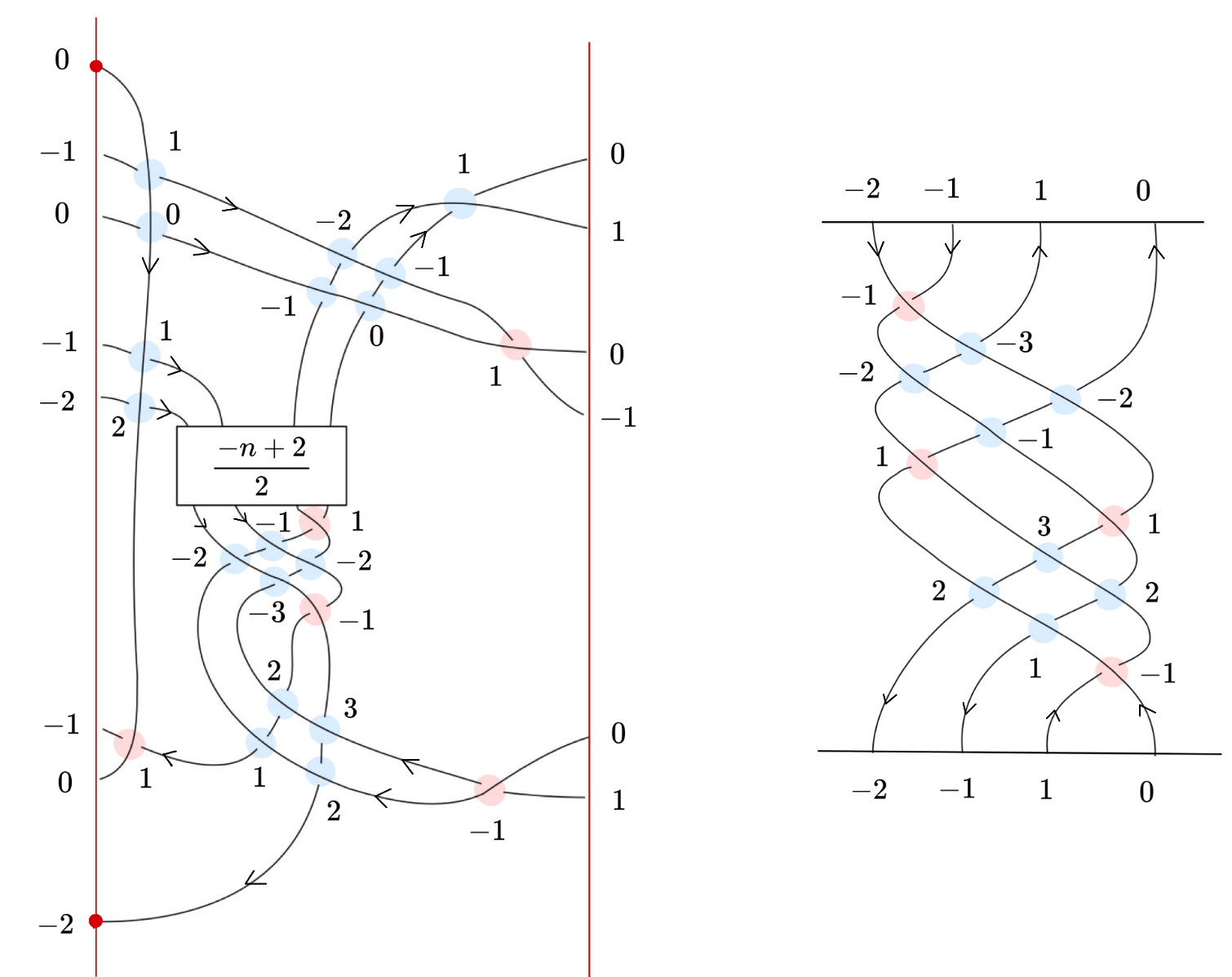}
\caption{Left: The pseudo-fundamental region of \((K_n,\sigma)\) where \(n\) is even. Right: A detail of one of the full twists in the box on the left.}
\label{sigma_even}
\end{center}
\end{figure}

\begin{figure}[h!]
\begin{center}
\includegraphics[width=6.0in]{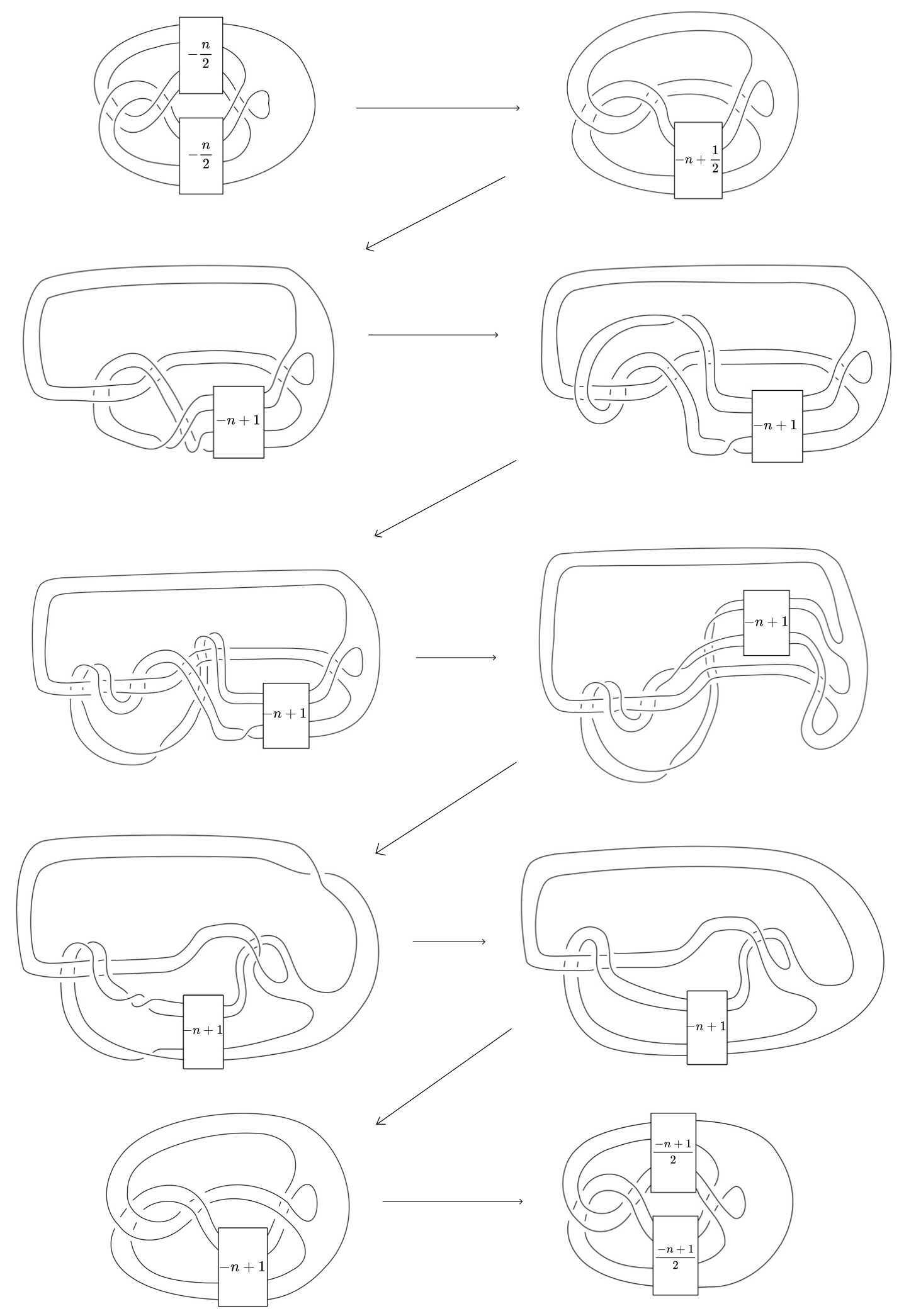}
\caption{The isotopy of \(K_n\) used in the proof of Lemma \ref{A1}}
\label{isotopy}
\end{center}
\end{figure}

\clearpage

\bibliographystyle{amsalpha}
\bibliography{references}

\end{document}